\documentclass[11pt,a4paper]{amsart}
\usepackage{amsfonts}

\usepackage{}
\usepackage[mathscr]{eucal}
\usepackage{amssymb}
\usepackage{amsbsy}

\usepackage{CJK,CJKnumb}
\usepackage[CJKbookmarks,colorlinks,
            linkcolor=black,
            anchorcolor=black,
            citecolor=blue]{hyperref}
\usepackage{color,soul}              % 支持彩色
\usepackage{indentfirst}        %首行缩进宏包
\usepackage{latexsym,bm}        % 处理数学公式中和黑斜体的宏包
\usepackage{graphics,graphicx}
\usepackage{cases}
\usepackage{pifont}
\usepackage{txfonts}
\usepackage{xcolor}
\usepackage[all,knot,poly]{xy}
\usepackage[usenames,dvipsnames]{pstricks}
\usepackage{epsfig}
\usepackage{pst-grad} % For gradients
\usepackage{pst-plot} % For axes

\allowdisplaybreaks[4]  % \eqnarray如果很长，影响分栏、换行和分页
\numberwithin{equation}{section}

\newcommand{\al}{\alpha}
\newcommand{\be}{\beta}
\newcommand{\de}{\delta}
\newcommand{\De}{\Delta}

\newcommand{\emp}{\emptyset}
\newcommand{\ga}{\gamma}
\newcommand{\Ga}{\Gamma}
\newcommand{\la}{\lambda}

\newcommand{\La}{\Lambda}
\newcommand{\ot}{\otimes}

\newcommand{\Om}{\Omega}
\newcommand{\si}{\sigmaup}
\newcommand{\te}{\theta}
\newcommand{\Te}{\Theta}
\newcommand{\vt}{\vartheta}

\newcommand{\mB}{\mathcal{B}}
\newcommand{\mH}{\mathcal{H}}
\newcommand{\mK}{\mathcal{K}}
\newcommand{\mP}{\mathcal{P}}
\newcommand{\mS}{\mathcal{S}}
\newcommand{\mc}{\mathscr{C}}
\newcommand{\pp}{\mathscr{P}}
\newcommand{\sq}{\mathscr{S}}
\newcommand{\fc}{\mathfrak{C}}
\newcommand{\fp}{\mathfrak{P}}

\newcommand{\bN}{\mathbb{N}}

\newcommand{\bY}{\mathbb{Y}}
\newcommand{\bZ}{\mathbb{Z}}

\newcommand{\bk}{\mathbb{K}}
\newcommand{\mb}[1]{\mbox{#1}}

\newcommand{\ms}[1]{\mbox{\sffamily #1}}
\newcommand{\bs}[1]{{\scriptsize\mbox{#1}}}

\newcommand{\ti}[1]{{\tiny \mbox{#1}}}
\newcommand{\stt}[1]{{\scriptstyle #1}}

\newcommand{\ull}[1]{\underline{#1}}

\newcommand{\lan}{\langle}
\newcommand{\ran}{\rangle}
\newcommand{\lb}{\left(}
\newcommand{\rb}{\right)}
\newcommand{\rw}{\rightarrow}
\newcommand{\beq}{\begin{equation}}
\newcommand{\eeq}{\end{equation}}

\begin{document}

\newtheorem{theorem}{Theorem}[section]

\newtheorem{lem}[theorem]{Lemma}

\newtheorem{cor}[theorem]{Corollary}
\newtheorem{prop}[theorem]{Proposition}

\theoremstyle{remark}
\newtheorem{rem}[theorem]{Remark}

\newtheorem{defn}[theorem]{Definition}

\newtheorem{exam}[theorem]{Example}

\theoremstyle{conjecture}
\newtheorem{con}[theorem]{Conjecture}

\renewcommand\arraystretch{1.2}
%%%%%%%%%%%%%%%%%%%%%%%%%%%%%%%%%%%%%%%%%%%%%%%
%%%%%%%%%%%%%%%%%%%%%%%%%%%%%%%%%%%%%%%%%%%%%%%%%

\title[A noncommutative lift of Schur's Q-functions]{A lift of Schur's Q-functions to the peak algebra}

\author[Jing]{Naihuan Jing}
\address{Department of Mathematics, North Carolina State University, Raleigh, NC 27695-8205, USA}
\email{jing@math.ncsu.edu}

\author[Li]{Yunnan Li$^\ast$}
\address{School of Mathematical Sciences, South China University of Technology, Guangzhou 510640, China}
\email{scynli@scut.edu.cn}

\thanks{$\ast$ corresponding author.}

%\date\today

\begin{abstract}
We construct a lift of Schur's Q-functions to the peak algebra of the symmetric group,
called the noncommutative Schur Q-functions, and extract from them a new natural basis with several
nice properties such as the positive right-Pieri rule, combinatorial expansion, etc. Dually, we get a basis of the Stembridge algebra of peak functions refining Schur's P-functions in a simple way.

\end{abstract}

\maketitle

\section{Introduction}

The algebra of noncommutative symmetric functions (abbreviated as NSym) is the noncommutative lifting of that of symmetric functions (abbreviated as Sym) studied first in \cite{GKL}. It is proved in \cite{MR} that the graded Hopf dual of NSym is the quasisymmetric functions (abbreviated as QSym), and NSym is isomorphic  as Hopf algebras to the Solomon descent algebra of the symmetric group.
As an important nonsymmetric generalization of Sym, QSym was introduced by Gessel as a source of generating functions for $P$-partitions \cite{Ge}. Later Stembridge developed this theory further to define the peak quasisymmetric functions as the weight enumerators of all enriched $P$-partitions of chains, which refine the classical Schur's Q-functions \cite{Ste}. The Stembridge algebra $\mB$ of peak functions has been widely studied and found close
 relations to various topics in combinatorics, geometry and representation theory, including Eulerian enumeration \cite{BHW}, Schubert calculus \cite{BH,BMSW}, Kazhdan-Lusztig theory \cite{BC}, etc.

The interesting relations among these combinatorial Hopf algebras above show the specialty of
peak functions. The peak algebra $\mP$ of the symmetric group is naturally embedded into NSym, with the graded Hopf dual isomorphic to the Stembridge algebra $\mB$. Also, $\mP$ can serve as a Hopf quotient of NSym via the $(1-t)$-transform at $t=-1$ introduced in \cite{KLT}, so is the case for $\mB$ in QSym by duality. Moreover, it is shown in \cite{ABS} that $\mP$ is the terminal object in the category of combinatorial Hopf algebras satisfying the generalized Dehn-Sommerville relation (\ref{np}). The latter is also called the Euler relation
and was first derived in \cite{BB} from the flag $f$-vectors of a ranked Eulerian poset by considering its M\"{o}bius function.

%[Theorem 2.1]

The main result of this paper is to find a noncommutative lifting of Schur's Q-functions in the peak algebra. We call them the \textit{noncommutative Schur Q-functions} (abbreviated as NSQF). They are derived by a creation operator construction lifting the vertex operator defined by the first author in \cite{J1} to realize Schur's Q-functions. Under the forgetful map $\pi$ from NSym to Sym, their image is the raising operator expression of Schur's Q-functions. This method has been applied by Berg et al. in \cite{BBSSZ} to construct a noncommutative lift of Schur functions, called the immaculate basis, and also for modified Hall-Littlewood functions. However, we emphasize that the results in \cite{BBSSZ} can not be specialized at $t=-1$ to recover ours, just like the case for the vertex operators defined by the first author in \cite{J2} and those in \cite{J1}.

A novel %innovative
point of our work is that we can extract a new and natural basis in the peak algebra $\mP$ indexed by the so-called \textit{peak compositions} from the NSQF, and that the new basis has a positive right-Pieri rule (\ref{pie'}). The peak composition set naturally contains all strict partitions, which parameterize Schur's Q-functions. Furthermore, in contrast with the anti-symmetric relations sastified by Schur's Q-functions, the NSQF's obey more subtle relations, which are still mysterious to us.

Dually, we find a new basis in the Stembridge algebra $\mB$, called
the \textit{quasisymmetric Schur Q-functions} (abbreviated as QSQF),
 since they also nicely refine Schur's Q-functions as the peak
 functions do (see (\ref{ref}), (\ref{ref'})). Our QSQF's are positively expanded in monomial quasisymmetric functions. Moreover, several simple examples convince us that they potentially have a positive, integral and
 unitriangular expansion in peak functions (Conjecture \ref{conj}), which in turn implies a positive expansion in fundamental quasisymmetric functions (Prop. \ref{pos}). It is also worthy of mentioning that other interesting bases for $\mB$ have been found. In an unpublished work \cite{Hsi}, Hsiao defined a monomial-like basis with its dual corresponding to a family of flag-enumeration functionals on Eulerian posets. Recently, another basis has been constructed in \cite{BC} based on a new characterization of $\mB$, and this result was applied to obtain a simple and explicit combinatorial formula for the Kazhdan-Lusztig polynomials of a Coxeter group $W$.

In \cite{BHT}, Bergeron et al. provided the peak algebra $\mP$ and its dual $\mB$ a representation theoretic interpretation as the Grothendieck ring of the tower of the Hecke-Clifford algebras at $q=0$. In particular,
the peak functions are realized as certain characters of simple supermodules. Hence, if Conjecture \ref{conj} holds, then our QSQF's may also have a nice character realization for some special modules. We note that similar work has been successively done for the dual immaculate basis due to Berg et al. in \cite{BBSSZ1}.

The organization of the paper is as follows. In $\S 2$ we provide some notation, definitions and mutual relations for all combinatorial Hopf algebras concerned, including NSym, QSym, Sym, $\mP$ and $\mB$. In $\S 3$ we lift the vertex operator realization of Schur's Q-functions to the noncommutative level based on the $(1-t)$-transform $Q_r$'s at $t=-1$. Then we obtain the raising operator expression (\ref{ro}), a key relation (\ref{re}) and a positive right-Pieri rule (\ref{pie}) for the NSQF. In $\S 4$ we find a natural basis for the peak algebra from the NSQF. A reformulated right-Pieri rule (\ref{pie'}) and also a simple combinatorial expression of the $Q_\al$'s in terms of
the new basis are given. In the last section, we obtain the dual basis in QSQF for the Stembridge algebra.
For representation theoretic consideration, the positive expansion of QSQF's in terms of the peak functions is
%a further problem one would like to solve (Conjecture \ref{conj}).
studied and Conjecture \ref{conj} is offered for future work.
\section{Background}

\subsection{Notation and definitions}
Denote by $\bN$ (resp. $\bN_0$) the set of positive (resp. nonnegative) integers. Given any $m,n\in\bN,\,m\leq n$, let $[m,n]:=\{m,m+1,\dots,n\}$ and $[n]:=[1,n]$ for short. Let $\mc(n)$ be the set of compositions of $n$, consisting of ordered tuples of positive integers summed up to $n$.  We denote $\al\vDash n$ when $\al\in \mc(n)$. Let $\mc:=\bigcup\limits_{n\geq1}^.\mc(n)$. Given $\al=(\al_1,\dots,\al_r)\vDash n$, let $\ell(\al)=r$ be its length and define its associated \textit{descent set} as
\[D(\al)=\{\al_1,\al_1+\al_2,\dots,\al_1+\cdots+\al_{r-1}\}
\subseteq[n-1].\]
The refining order $\leq$ on $\mc(n)$ is defined by
\[\al\leq\be\mb{ if and only if }D(\be)\subseteq D(\al),\,\forall\al,\be\vDash n.\]
In general, for $\al\in\bN^r_0$, let $\ell(\al)=|\{i\,:\,\al_i>0\}|$.

We highlight the subset $\mc_o(n)$ of $\mc(n)$, consisting of compositions of $n$ with odd parts, and write $\al\vDash_{\bs{odd}} n$ when $\al\in\mc_o(n)$.
%Let $\mc_o=\bigcup\limits_{n\geq1}^.\mc_o(n)$ and $o_n:=|\mc_o(n)|$, then
It is well-known that
\[|\mc_o(n)|=f_{n-1},\]
where $\{f_n\}_{n\geq0}$ is the Fibonacci sequence defined recursively by
\[f_0=f_1=1,\,f_n=f_{n-1}+f_{n-2},\,n\geq2.\]

Now fix an algebraically closed field $\bk$ of characteristic 0. Given the alphabet $A=\{a_1,a_2,\dots\}$, one has the free associative $\bk$-algebra $\mathcal {F}=\bk\lan\lan a_1,a_2,\dots\ran\ran$. Define the following functions
in the ring $\mathcal {F}[[z]]$ of $\mathcal F$-power series in the variable $z$:
\begin{equation}\label{ncom}
H(A,z)=\sum_{n\geq0}H_n(A)z^n=\prod\limits_{i\geq1}^{\longrightarrow}\dfrac{1}{1-a_iz}
\end{equation}
and
\begin{equation}\label{nele}
E(A,z)=\sum_{n\geq0}E_n(A)z^n=\prod\limits_{i\geq1}^{\longleftarrow}(1+a_iz).
\end{equation}
It is easy to see that
\[\sum_{i=0}^n(-1)^iE_i(A)H_{n-i}(A)=\de_{n,0}.\]
The set $\{H_n(A)\}_{n\in\mathbb{N}}$ (or $\{E_n(A)\}_{n\in\mathbb{N}}$) generates a subalgebra of $\mathcal {F}$, called the algebra of \textit{noncommutative symmetric functions} and denoted by NSym \cite[\S 4]{KLT}. The algebra
$\mb{NSym}=\oplus_{n=0}^{\infty} \mb{NSym}_n$ is a $\mathbb Z$-graded algebra
under the gradation given by $\mb{deg}(H_n)=n$, where $\mb{NSym}_n$ is
the subspace of homogeneous elements of degree $n$.
Let \[H_\al=H_{\al_1}\cdots H_{\al_r},E_\al=E_{\al_1}\cdots E_{\al_r},~\al=(\al_1,\dots,\al_r)\vDash n.\]
If we change base from $\bk$ to $\bZ$, then both $\{H_\al\}_{\al\vDash n}$ and $\{E_\al\}_{\al\vDash n}$ are $\bZ$-bases of $\mb{NSym}_n$, called the {\it noncommutative complete} and {\it elementary symmetric functions} respectively. There exists another important $\bZ$-basis $\{R_\al\}_{\al\vDash n}$ of $\mb{NSym}_n$, called the \textit{noncommutative ribbon Schur functions} and are defined by
\[R_\al=\sum_{\be\geq\al}(-1)^{l(\be)-l(\al)}H_\be.\]

%One can define the noncommutative power-sums $\Psi_n$ of the first kind by
% \[\Psi(A,z)=\sum_{n\geq1}\Psi_nz^{n-1},
% \,\dfrac{d}{dz}H(A,z)=H(A,z)\Psi(A,z).\]
%Then we also have
%\[\dfrac{d}{dz}E(A,-z)=-\Psi(A,z)E(A,-z).\]
%\begin{rem}
%Note that
%\[H_n=\sum_{i_1\leq\cdots\leq i_n}a_{i_1}\cdots a_{i_n},\,
%E_n=\sum_{i_1<\cdots<i_n}a_{i_1}\cdots a_{i_n}.\]
%However, it should be a terrible mistake to believe that
%\[\Psi_n=\sum_{i\geq1}a_i^n.\]
%For instance, by the Newton relation $\sum_{k=1}^n(-1)^{k-1}\Psi_kE_{n-k}=nE_n$, we have
%\[\begin{split}
%\Psi_3&=3E_3-2E_{21}-E_{12}+E_{111}\\
%&=\sum_ia_i^3+\sum_{i<j<k}(a_ia_ja_k+a_ka_ja_i-a_ia_ka_j-a_ja_ka_i)
%+\sum_{i<j}(a_ja_i^2+a_j^2a_i-a_ia_j^2-a_ia_ja_i).
%\end{split}\]
%Only when $a_i$'s commute, we can cancel terms to get the usual power-sums.
%\end{rem}
%
%\smallskip
%Note that $H(A,z),\,E(A,z)$ are invertible in $\mathcal {F}[[z]]$ and $E(A,-z)H(A,z)=1$, thus also $H(A,z)E(A,-z)=1$, or equivalently we have
%\begin{equation}\label{hen}
%H_0=E_0=1,~\sum\limits_{k=0}^n(-1)^kE_kH_{n-k}=
%\sum\limits_{k=0}^n(-1)^kH_kE_{n-k}=0,~n\in\mathbb{N}.
%\end{equation}

Let
 \[Q(A,z)=\sum_{n\geq0}Q_n(A)z^n=E(A,z)H(A,z)=
 \prod\limits_{i\geq1}^{\longleftarrow}(1+a_iz)
 \prod\limits_{i\geq1}^{\longrightarrow}\dfrac{1}{1-a_iz}.\]
Write $Q_n(A)$ as $Q_n$ for short and let $Q_\al=Q_{\al_1}\cdots Q_{\al_r},~\al=(\al_1,\dots,\al_r)\vDash n$, then
\[Q_0=1,\,Q_n=\sum_{k=0}^nE_kH_{n-k},\,n\geq1.\]
We also let $Q_n=0,\,n<0$ for convenience. It should be noticed that $Q(A,z)\neq\prod\limits_{i\geq1}^{\longleftarrow}\dfrac{1+a_iz}{1-a_iz}$ or $\prod\limits_{i\geq1}^{\longrightarrow}\dfrac{1+a_iz}{1-a_iz}$. By definition it follows that $Q(A,z)Q(A,-z)=1$, that is, the Euler relations
\beq\label{np}\sum_{r+s=n}(-1)^rQ_rQ_s=0,\,\forall n\geq1\eeq
hold, or equivalently,
\beq\label{eo}\sum_{i=1}^{n-1}(-1)^{i-1}Q_iQ_{n-i}=\begin{cases}
2Q_n, & $if $ n $ is even$,\\
0, & $if $ n $ is odd$.
\end{cases}\eeq
When $Q_r$'s commute, the odd part of identity (\ref{eo}) is trivial. For the noncommutative case, it can also be deduced from the even part. In fact, iterative use of the even part implies that
\begin{prop}
When $n$ is even, we have
\beq\label{even}
Q_n=\sum_{\al\vDash_{\ti{odd}} n}(-1)^{\ell(\al)/2-1}C_{\ell(\al)/2-1}2^{-\ell(\al)+1}Q_\al.
\eeq
where $C_k=\tfrac{1}{k+1}{2k\choose k}\,(k\geq0)$, the $k$th Catalan number.
\end{prop}
\begin{proof}
Fix $\al\vDash_{\bs{odd}} n$ with odd parts. In order to divide $n$ into $\al$, we can first divide $n$ into $\ell(\al)/2$ even parts, and then split each part into two odd ones to get $\al$. In our case, the first step contributes coefficient $C_{\ell(\al)/2-1}(-2)^{-\ell(\al)/2+1}$, while the second step gives $2^{-\ell(\al)/2}$. They combine to give the desired coefficient of $Q_\al$ on the RHS of (\ref{even}).
\end{proof}

\noindent
Now for $n=2k+1,\,k\geq0$,
\[\sum_{i=1}^{n-1}(-1)^{i-1}Q_iQ_{n-i}=
-\sum_{i=0}^{k-1}Q_{2i+1}Q_{2(k-i)}+
\sum_{i=1}^{k}Q_{2i}Q_{2(k-i)+1}.\]
If we expand those $Q_r$'s with $r$ even by (\ref{even}), the terms will cancel in pairs.

%\begin{rem}
%In contrast with the commutative case, $Q_r$'s with $r$ odd are no longer generated by those power-sum $\Psi_r$'s with $r$ odd only. For instance,
%\[\begin{split}
%Q_3&=H_3+E_1H_2+E_2H_1+E_3=(\dfrac{1}{3}\Psi_3+\dfrac{1}{6}\Psi_{21}
%+\dfrac{1}{3}\Psi_{12}+\dfrac{1}{6}\Psi_{111})\\
%&+\Psi_1(\dfrac{1}{2}\Psi_2+\dfrac{1}{2}\Psi_{11})
%+(-\dfrac{1}{2}\Psi_2+\dfrac{1}{2}\Psi_{11})
%\Psi_1+(\dfrac{1}{3}\Psi_3-\dfrac{1}{3}\Psi_{21}
%-\dfrac{1}{6}\Psi_{12}+\dfrac{1}{6}\Psi_{111})\\
%&=\dfrac{2}{3}\Psi_3+\dfrac{4}{3}\Psi_{111}+\dfrac{2}{3}(\ull{\Psi_{12}
%-\Psi_{21}}).
%\end{split}\]
%
%\end{rem}

%Equivalently the generating function of $\{o_n\}$ is given by
%\[\sum_{n\geq1}o_nz^n=\dfrac{t}{1-z-z^2}.\]

%We remark that the paper \cite{And} gives a bijective proof that $oc_n$ equals the number of compositions of $n+1$ with parts greater than one.
\medskip
In general, let $t$ be an indeterminate, $\bk(t)$ be the rational field and $\mathcal {F}(t):=\bk(t)\lan\lan a_1,a_2,\dots\ran\ran$. Define the generating sequence in $\mathcal{F}(t)[[z]]$,
 \[Q(A,t,z):=\sum_{n\geq0}Q_n(A,t)z^n=E(A,-tz)H(A,z)=
 \prod\limits_{i\geq1}^{\longleftarrow}(1-a_itz)
 \prod\limits_{i\geq1}^{\longrightarrow}\dfrac{1}{1-a_iz}.\]

Note that $Q_n(A,0)=H_n(A)$, thus
$\{Q_\al(t):=Q_\al(A,t)\}_{\al\in \mc}$ also forms a $\bZ[t]$-basis of $\ms{NSym}$. We remark that $Q_n(A,t)$ is just the $(1-t)$-transform $H_n((1-t)A)$ of $H_n(A)$ discussed in \cite[\S 5]{KLT}. Moreover, by \cite[Eq.(68)]{KLT}
\beq\label{coq}
\De(Q_n(t))=\sum_{k=0}^n Q_k(t)\ot Q_{n-k}(t).\eeq
In particular, $Q_n(A)$ is just the $(1-t)$-transform of $Q_n(A,t)$ at $t=-1$ (see also \cite[\S 2]{BHT}).

% and there exists a canonical bilinear form
%  \begin{equation}\label{bi}
%  \lan\cdot,\cdot\ran:\ms{QSym}\times\ms{NSym}\rightarrow \bk(t),~\lan M_\al,Q_\be\ran=\de_{\al,\be},~\al,\be\in C.
%  \end{equation}
%Correspondingly, we have the following noncommutative Cauchy identity:
%\[\prod_{i\geq1}^{\longrightarrow}\lb\prod\limits_{j\geq1}^{\longleftarrow}(1+tx_ia_j)
% \prod\limits_{j\geq1}^{\longrightarrow}\dfrac{1}{1-x_ia_j}\rb
% =\sum_\al M_\al(X) Q_\al(A,t).\]

%Define the noncommutative Jing operators $\bJ_m:\ms{NSym}\rw\ms{NSym}$ as
%\[\bJ_m=\sum_{i,j\geq0}(-1)^it^jQ_{m+i+j}(t)F_{1^i}^\bot F_j^\bot.\]

\subsection{The peak subalgebra and its Hopf dual}
Let $\mP$ be the Hopf subalgebra of NSym generated by $Q_n\,(n\geq1)$. Then $\mP_n:=\mP\cap\mb{NSym}_n$ is isomorphic to the \textit{peak algebra} of the symmetric group $\mathfrak{S}_n$ when endowed with the internal product \cite{BHT,Sch}. According to (\ref{coq}), one can define a surjective Hopf algebra homomorphism
\[\Te:\mb{NSym}\rw\mP,\quad H_n\mapsto Q_n,\,n\geq1.\]
From \cite[Theorem 5.4]{BMSW1}, we know that $\mb{Ker }\Te$ is the Hopf ideal of NSym generated by
\[\mH_{2n}:=\sum_{i+j=2n}(-1)^iH_iH_j,\,n\geq1,\]
which correspond to the (even) Euler relations (\ref{np}). Equivalently,
$\{Q_\al\}_{\al\vDash_{\ti{odd}}n}$ forms a linear basis of $\mP_n$, according to \cite[Main Theorem 3]{Sch}.

It is well-known that the graded Hopf dual of NSym is the algebra of \textit{quasisymmetric functions}, denoted by QSym \cite{MR}.
It is a subring of the power series ring $\bk[[x_1,x_2,\dots]]$ in the commuting variables $x_1,x_2,\dots$ and has a linear basis, the \textit{monomial quasisymmetric functions}, defined by
\[M_\al:=M_\al(x)=\sum\limits_{i_1<\cdots< i_r}x_{i_1}^{\al_1}\cdots x_{i_r}^{\al_r},\]
where $\al=(\al_1,\dots,\al_r)$ varies over the set $\mc$ of compositions. There is another important basis, the \textit{fundamental quasisymmetric functions}, defined by
\[F_\al:=F_\al(x)=\sum\limits_{i_1\leq\cdots\leq i_n\atop i_k<i_{k+1}\mb{ \tiny if }k\in D(\al)}x_{i_1}\cdots x_{i_n},\,\al\vDash n.\]
In other words, $F_\al=\sum_{\be\leq\al}M_\be$. Meanwhile, the canonical pairing $\lan\cdot,\cdot\ran$ between NSym and QSym is defined by
\[\lan H_\al,M_\be\ran=\lan R_\al,F_\be\ran=\de_{\al,\be}\]
for any $\al,\be\in\mc$.

Let $\La$ be the graded ring of symmetric functions in the commuting variables $x_1,x_2,\dots$, with integer coefficients, and $\Om$ be the subring of $\La$ generated by the symmetric functions $q_n\,(n\geq1)$, which are defined by
\[\sum_{n\geq0}q_nz^n=\prod_{i\geq1}\dfrac{1+x_iz}{1-x_iz}.\]
For the basics of this subring $\Om$ and Schur's Q-functions, one can refer to \cite[Ch. III, \S 8]{Mac}, where $\Om$ is denoted as $\Ga$.
There exists a Hopf algebra epimorphism
\[\te:\La\rw\Om,\quad h_n\mapsto q_n,\,n\geq1.\]
Then $\te(p_n)=(1-(-1)^n)p_n,\,n\geq1$, where $p_n$'s are the power-sum symmetric functions. Also let \[\pi:\mb{NSym}\rw\La,\,H_n\mapsto h_n\]
be the \textit{forgetful map}.

Now we introduce the famous \textit{Stembridge algebra} $\mB$ of peak functions defined in \cite{Ste}. This is a Hopf subalgebra of QSym.
In order to define the usual bases of $\mP$ and $\mB$, we recall the concept of peak subsets of $[n]$. A subset $P\subseteq[n]$ is called a \textit{peak set} in $[n]$ if $P\subseteq[2,n-1]$ and $i\in P\Rightarrow i-1\notin P$. Denote by $\pp_n$ the collection of peak sets in $[n]$, $\pp:=\bigcup\limits^._{n\geq1}\pp_n$, and $\emp_n$ the empty set $\emp$ in $\pp_n$. Given $\al=(\al_1,\dots,\al_r)\vDash n$, let
\[P(\al):=\{x\in[2,n-1]\,:\,x\in D(\al),\,x-1\notin D(\al)\}\]
be its associated peak set in $[n]$. For any $P\in\pp_n$, define
\beq\label{pi}\Pi_P=\sum_{P(\al)=P}R_\al\in\mb{NSym}.\eeq
Then $\{\Pi_P\}_{P\in\pp_n}$ forms a linear basis of $\mP_n$ \cite[\S 2]{BHT}. Note that by \cite[Eq.(6)]{BHT},
\[Q_n=2\Pi_{\emp_n}=2\sum_{k=0}^{n-1}R_{1^k,n-k},\,n\geq1.\]

On the other hand, Stembridge's \textit{peak functions} in $\mB$ can be defined by \cite[Prop. 3.5]{Ste}
\beq\label{ka}K_P:=2^{|P|+1}\sum_{\al\vDash n\atop P\subseteq D(\al)\triangle(D(\al)+1)}F_\al,\,P\in\pp_n,\eeq
where $D\triangle(D+1)=D\backslash(D+1)\cup (D+1)\backslash D$ for any $D=\{D_1<\cdots<D_r\}\subseteq[n-1]$ and $D+1:=\{x+1\,:\,x\in D\}$. Then $\{K_P\}_{P\in\pp_n}$ forms a linear basis of $\mB_n$ and there also exists a surjective Hopf algebra homomorphism
\[\vt:\mb{QSym}\rw\mB,\quad F_\al\mapsto K_{P(\al)}.\]
The coproduct formula of peak functions can be found in \cite[Lemma 1.4]{BMSW}. By \cite[Prop. 2.2]{Ste},
\[K_P=\sum_{\al\vDash n\atop P\subseteq D(\al)\cup(D(\al)+1)}2^{\ell(\al)}M_\al,\,P\in\pp_n\]
and in particular, by \cite[(2.5)]{Ste},
\beq\label{ke}
K_{\emp_n}=q_n=2\sum_{\al\vDash n}F_\al=\sum_{\al\vDash n}2^{\ell(\al)}M_\al.\eeq

The Hopf subalgebra $\mP$ can be regarded as a noncommutative lift of $\Om$, in which we shall find a lift of Schur's Q-functions. The following commutative diagrams illustrate the situation.
\[\xymatrix@=2em{\mb{NSym}\ar@{->}[r]^-{\Te}\ar@{->}[d]_-{\pi}
&\mP\ar@{->}[d]^-{\pi}\\
\La\ar@{->}[r]^-{\te}&\Om},\quad
\xymatrix@=2em{\mb{QSym}\ar@{->}[r]^-{\vt}
&\mB\\
\La\ar@{->}[r]^-{\te}\ar@{->}[u]&\Om\ar@{->}[u]},\]
where the vertical maps in the second diagram are inclusions.

In fact, one can define a graded Hopf dual pairing
\[[\cdot,\cdot]:\mP\times\mB\rw\bk,\quad [\Pi_P,K_Q]=\de_{P,Q},\,P,Q\in\pp,\]
which satisfies the following property \cite[Cor. 5.6.]{Sch},
\[\lan\Te(F),f\ran=\lan F,\vt(f)\ran=[\Te(F),\vt(f)],\, F\in\mb{NSym},f\in\mb{QSym}.\]
In particular, when $f\in\La$, then
\beq\label{bi}[\Te(F),\te(f)]=\lan\Te(F),f\ran
=\lan\pi\Te(F),f\ran=[\pi\Te(F),\te(f)],\eeq
where the rightmost one is the canonical inner product $[\cdot,\cdot]$ on $\Om$ defined by
\[[p_\la,p_\mu]=z_\la2^{-\ell(\la)}\de_{\la,\mu}\]
for any strict partitions $\la,\mu$.

\section{A noncommutative lift of Schur's Q-functions}

Now we are in the position to give our main construction. Given $f\in\mB$, we define the adjoint operator $f^\perp\in\mb{End}(\mP)$ by  \[[f^\perp(H),g]=[H,fg]\]
for any $H\in\mP,\,g\in\mB$. Similarly for $f\in\Om$, define  $f^\perp\in\mb{End}(\Om)$ by
\[[f^\perp(h),g]=[h,fg],\,g,h\in\Om.\]

\begin{defn}
We define the formal power series $\bY(z)$ in $\mb{End}(\mP)[[z,z^{-1}]]$ via
\beq\bY(z)=\sum_{n\in\bZ}\bY_n z^{-n}=
\lb\sum_{n\geq0}Q_nz^n\rb\lb\sum_{n\geq0}K_{\emp_n}^\perp(-z)^{-n}\rb,\eeq
i.e.
\begin{equation}\label{e:nY}
\bY_n=\sum_{i\geq0}(-1)^iQ_{-n+i}K_{\emp_i}^\perp,\,n\in\bZ.
\end{equation}

Note that (\ref{e:nY}) is well-defined, as $\bY_n$ has a finite expansion when acting on the linear basis $Q_{\al}$,
$\al\vDash_{\ti{odd}}n$ (see Lemma \ref{kq}). For $\al=(\al_1,\dots,\al_r)\in\bZ^r$, define
\[\sq_\al:=\bY_{-\al}(1)=\bY_{-\al_1}\cdots\bY_{-\al_r}(1).\]
\end{defn}

\smallskip
The following lemma shows that $\sq_\al$'s are a lift of Schur's Q-functions onto NSym, thus are called the \textit{noncommutative Schur Q-functions} (abbreviated as NSQF).
\begin{lem}
(1) For any $f\in\Om\subset\mB$, $\pi f^\perp=f^\perp \pi$.

(2) $\pi\bY(z)=Y(z)\pi$, where $Y(z)$ is the twisted vertex operator on $\Om$ defined in \cite{J1},
\[Y(z)=\sum_{n\in\bZ}Y_n z^{-n}=
\lb\sum_{n\geq0}q_nz^n\rb\lb\sum_{n\geq0}q_n^\perp(-z)^{-n}\rb=
\mb{exp}\lb\sum_{n\in\bN_{\bs{odd}}}\dfrac{2p_n}{n}z^n\rb
\mb{exp}\lb-\sum_{n\in\bN_{\bs{odd}}}\dfrac{2p_n^\perp}{n}z^{-n}\rb.\]

(3) For any $\al\in\mc$, $\pi(\sq_\al)=\mS_\al$, where
$\mS_\al$ is the Schur Q-function indexed by $\al$.
\end{lem}

\begin{proof}
For (1), according to (\ref{bi}),
\[[\pi f^\perp(H),g]=[f^\perp(H),g]=[H,fg]=[\pi(H),fg]=[f^\perp \pi(H),g]\]
for any $H\in\mP,\,g\in\Om$. Hence, $\pi f^\perp=f^\perp \pi$. Combining (1) with the identities $\pi(Q_n)=q_n,\,n\geq0$ and (\ref{ke}), one gets (2).

On the other hand,  we know that  $\mS_\al=Y_{-\al}(1)$ by \cite[Theorem 5.9]{J1}. Moreover, by \cite[Prop. 4.15]{J1} we have
\[\{Y_n,Y_m\}=Y_nY_m+Y_mY_n=(-1)^n2\de_{n,-m},\,n,m\in\bZ.\]
Hence, one gets the basis of Schur's Q-functions indexed by strict partitions. Now (3) follows from (2).
\end{proof}

\begin{lem}\label{kq}
For any $n\geq1$ and $\al=(\al_1,\dots,\al_r)\in\mc$, we have
\[K_{\emp_n}^\perp(Q_\al)=\sum_{\be\in\bN_0^r\atop |\be|=n}2^{\ell(\be)}Q_{\al-\be}.\]
\end{lem}
\begin{proof}
Since $[\cdot,\cdot]$ is a Hopf dual pair, one can easily see that (for example, \cite[Lemma 2.4]{BBSSZ})
\[f^\perp(GH)=\sum f_{(1)}^\perp(G)f_{(2)}^\perp(H)\]
for any $f\in\mB$ and $G,H\in\mP$, and $\Delta(f)=\sum f_{(1)}\otimes f_{(2)}$
in Sweedler's notation.
Meanwhile, for any $f\in\mB$,
\[[K_{\emp_n}^\perp(Q_m),f]=[Q_m,K_{\emp_n}f]=\sum_{i=0}^m
[Q_i,K_{\emp_n}][Q_{m-i},f]=2^{1-\de_{n,0}}[Q_{m-n},f]\]
as $Q_m=2^{1-\de_{m,0}}\Pi_{\emp_m}$. It means that
\[K_{\emp_n}^\perp(Q_m)=2^{1-\de_{n,0}}Q_{m-n}.\]
In particular, the formula
\[\De^{(r-1)}(K_{\emp_n})=\sum_{\be\in\bN_0^r\atop |\be|=n}
K_{\emp_{\be_1}}\ot\cdots\ot K_{\emp_{\be_r}},\,r\geq2\]
gives the desired result.
\end{proof}

%\begin{lem}
%For any $m\in\bZ,\,r\in\bN_0$,
%
%(1)\quad $K_{\emp_r}^\perp\bY_{-m}=\bY_{-m}K_{\emp_r}^\perp
%+2\sum\limits_{i=1}^r\bY_{-m+i}K_{\emp_{r-i}}^\perp$.
%
%(2)\quad $\sum\limits_{i\geq0}\bY_{-m-i}K_{\emp_i}^\perp=Q_m$.
%\end{lem}
%\begin{proof}
%For (1), by definition and Lemma \ref{kq}
%\[\begin{split}
%K_{\emp_r}^\perp\bY_{-m}&=K_{\emp_r}^\perp\sum_{i\geq0}(-1)^i Q_{m+i}K_{\emp_i}^\perp=\sum_{i\geq0}\sum_{j=0}^r(-1)^i K_{\emp_j}^\perp(Q_{m+i})K_{\emp_{r-j}}^\perp K_{\emp_i}^\perp\\
%&=\sum_{j=0}^r 2^{1-\de_{j,0}}\lb\sum_{i\geq0}(-1)^i Q_{m-j+i}K_{\emp_i}^\perp\rb K_{\emp_{r-j}}^\perp=\sum_{j=0}^r 2^{1-\de_{j,0}}\bY_{-m+j}K_{\emp_{r-j}}^\perp.
%\end{split}\]
%
%For (2), by definition
%\[\sum\limits_{i\geq0}\bY_{-m-i}K_{\emp_i}^\perp=
%\sum_{i\geq0}\lb\sum_{j\geq0}(-1)^j Q_{m+i+j}K_{{\emp_j}}^\perp\rb K_{\emp_i}^\perp=
%\sum_{k\geq0} Q_{m+k}\lb\sum_{j=0}^k(-1)^jK_{{\emp_j}}^\perp K_{\emp_{k-j}}^\perp\rb=Q_m,\]
%where we use the Euler relation for $K_{\emp_n}=q_n$ to obtain the last equality.
%\end{proof}

\begin{defn}
For an ordered set $X=\{x_1,x_2,\dots\}$ of commuting variables which also commute with the letters $a_1,a_2,\dots$ in $A$, we define the noncommutative analogue of the \textit{Cauchy kernel} associated with Schur's Q-functions.
\[\Xi_X:=\prod_{i\geq1}\lb\prod\limits_{j\geq1}^{\longleftarrow}(1+x_ia_j)
 \prod\limits_{j\geq1}^{\longrightarrow}\dfrac{1}{1-x_ia_j}\rb=\sum_{\al\in \mc}M_\al(X)Q_\al(A)=\sum_{P\in\pp}K_P(X)\Pi_P(A).\]
\end{defn}
\noindent In particular, when $X$  has only one variable $z$, $\Xi_z=\sum_{n\geq0}Q_n(A)z^n$.

\begin{lem}\label{ck}
The noncommutative Cauchy kernel $\Xi_X$ has the following properties:

(1) $\Xi_{z,X}=\Xi_z\Xi_X$, where $z,X$ denotes the alphabet $\{z,x_1,x_2,\dots\}$.

(2) For any $f\in\mB$, $f^\perp(\Xi_X)=f(X)\Xi_X$, where $f^\perp$ only acts on the quasisymmetric part of $\Xi_X$.
\end{lem}
\begin{proof}
(1) is clear by the definition of $\Xi_X$. For (2),
\[\begin{split}
f^\perp(\Xi_X)&=\sum_{P\in\pp}K_P(X)f^\perp(\Pi_P)
=\sum_{P\in\pp}K_P(X)\sum_{Q\in\pp}[\Pi_P,fK_Q]\Pi_Q\\
&=\sum_{Q\in\pp}(fK_Q)(X)\Pi_Q=f(X)\Xi_X.
\end{split}\]
\end{proof}

In particular, let $\mK_z^\perp=\sum_{i\geq0}z^iK_{\emp_i}^\perp$. Then from Lemma \ref{ck} (2), we have
\[\mK_z^\perp\Xi_X=\sum_{i\geq0}z^iK_{\emp_i}(X)\Xi_X=\Xi_X\prod_{x\in X}\dfrac{1+zx}{1-zx}.\]

Using the notations above, we know that
\[\bY(z)=\lb\sum_{n\geq0}Q_nz^n\rb\lb\sum_{n\geq0}K_{\emp_n}^\perp(-z)^{-n}\rb
=\Xi_z\mK_{-1/z}^\perp.\]
Hence,
\[\bY(z)\Xi_X=\Xi_z\mK_{-1/z}^\perp\Xi_X=\Xi_z\Xi_X\prod_{x\in X}\dfrac{1-x/z}{1+x/z}=\Xi_{z,X}\prod_{x\in X}\dfrac{1-x/z}{1+x/z}.\]
More generally,
\[\bY(z_1)\cdots\bY(z_r)\Xi_X=\Xi_{z_1,\dots,z_r,X}\prod_{i=1}^r\lb\prod_{x\in \{z_{i+1},\dots,z_r\}\cup X}\dfrac{1-x/z_i}{1+x/z_i}\rb.\]
Now let $X=\emp$ and take the coefficient of $z_1^{\al_1}\cdots z_r^{\al_r}$ in the above series, we get
\begin{prop}\label{ras}
For $\al=(\al_1,\dots,\al_r)\in\bZ^r$,
\beq\label{ro}\sq_\al:=\bY_{-\al_1}\cdots\bY_{-\al_r}(1)=\prod_{1\leq i<j\leq r}\dfrac{1-R_{ij}}{1+R_{ij}}Q_\al,\eeq
where $R_{ij}Q_{\al}=Q_{R_{ij}\al}$ and $R_{ij}$ is the usual raising operator acting on $\bZ^r$ by
\[R_{ij}(\al_1,\dots,\al_r)=(\al_1,\dots,\al_i+1,\dots,\al_j-1,\dots,\al_r).\]
\end{prop}

\begin{cor}
For $\al=(\al_1,\dots,\al_r)\in\mc$, let $\al':=(\al_1,\dots,\al_{r-1})$, then
$\sq_\al$ satisfies the following recursive relation
\beq\label{rec}\sq_\al=\sum_{i=0}^{\al_r}\lb\sum_{\be\in\bN_0^{r-1}\atop |\be|=i}
\prod_{j=1}^{r-1}(-1)^{\be_j}2^{1-\de_{\be_j,0}}\sq_{\al'+\be}\rb Q_{\al_r-i}.\eeq
\end{cor}
\begin{proof}
Applying the identity \[\prod_{i=1}^{r-1}\dfrac{1-R_{ir}}{1+R_{ir}}=\prod_{i=1}^{r-1}
\lb 1+2\sum_{k\geq1}(-1)^kR_{ir}^k\rb\]
to the right hand side of formula (\ref{ro}), we get the desired recursive relation.
\end{proof}

%Meanwhile, we can give a Pfaffian expression of NSQF, lifting the case for Schur's Q-functions.
In order to find a concrete expansion of NSQF in $Q_\al$'s, we recall the following notation. For any anti-symmetric matrix $A=(a_{ij})_{1\leq i,j\leq 2n}$ over a commutative ring $R$, the \textit{Pfaffian} $\mb{Pf}(A)\in R$ is the unique square root of det$(A)$ up to a sign defined by
\[\mb{Pf}(A):=\sum_{\si\in\mathfrak{S}'_{2n}}(-1)^{\ell(\si)}a_{\si(1),\si(2)}\cdots
a_{\si(2n-1),\si(2n)},\]
where $\mathfrak{S}'_{2n}=\{\si\in\mathfrak{S}_{2n}:\si(2i-1)<\si(2i),\,1\leq i\leq n;\si(2j-1)<\si(2j+1),\,1\leq j\leq n-1\}$.

%\begin{prop}
%For $\al=(\al_1,\dots,\al_{2n})\in\bZ^{2n}$,
%\beq\label{pf}\sq_\al=\mb{\upshape Pf}(\sq_{\al_i,\al_j})_{1\leq i,j\leq 2n}.\eeq
%\end{prop}
From Prop. \ref{ras}, we know that for any $\al=(\al_1,\dots,\al_{2n})\in\bZ^{2n}$,
$\sq_\al$ is the coefficient of $z^\al$ in
\[\begin{split}
&\,\Xi_{z_1,\dots,z_{2n}}\prod_{1\leq i<j\leq {2n}}\lb\dfrac{1-z_j/z_i}{1+z_j/z_i}\rb=
\Xi_{z_1,\dots,z_{2n}}
\mb{Pf}\lb\dfrac{z_i-z_j}{z_i+z_j}\rb\\
&=\sum_{\si\in\mathfrak{S}'_{2n}}(-1)^{\ell(\si)}
\lb\dfrac{z_{\si(1)}-z_{\si(2)}}
{z_{\si(1)}+z_{\si(2)}}\rb\cdots\lb\dfrac{z_{\si(2n-1)}-z_{\si(2n)}}
{z_{\si(2n-1)}+z_{\si(2n)}}\rb \sum_{\be\in\bN^{2n}}
z_1^{\be_1}\cdots z_{2n}^{\be_{2n}} Q_{\be_1}\cdots Q_{\be_{2n}}\\
&=\sum_{\be\in\bN^{2n}}\sum_{\si\in\mathfrak{S}'_{2n}}
\sum_{i_1,i_3\dots,i_{2n-1}\leq0\atop i_{2k-1}+i_{2k}=0}(-1)^{\ell(\si)+\sum_{k=1}^ni_{2k-1}}
2^{n-\sum_{k=1}^n\de_{i_{2k-1},0}}
z_{\si(1)}^{\be_{\si(1)}+i_1}\cdots z_{\si(2n)}^{\be_{\si(2n)}+i_{2n}}
Q_{\be_1}\cdots Q_{\be_{2n}}\\
&=\sum_{\be\in\bN^{2n}}\sum_{\si^{-1}\in\mathfrak{S}'_{2n}}
\lb\sum_{i_1,i_3\dots,i_{2n-1}\leq0\atop i_{2k-1}+i_{2k}=0}(-1)^{\ell(\si)+\sum_{k=1}^ni_{2k-1}}
2^{n-\sum_{k=1}^n\de_{i_{2k-1},0}}\rb
z_1^{\be_1+i_{\si(1)}}\cdots z_{2n}^{\be_{2n}+i_{\si(2n)}}
Q_{\be_1}\cdots Q_{\be_{2n}},
\end{split}\]
where the first equality for the Pfaffian can be found in \cite[Ch. III, \S 8, Ex. 5]{Mac}.
This proves the following nontrivial expansion of $\sq_\al$ in terms of $Q_{\be}$'s.

\begin{prop} \label{P:iter} For any $\al\in\mathbb Z^{2n}$, we have that
\begin{equation}
\sq_{\al}=\sum_{\si\in\mathfrak{S}'_{2n}}
\lb\sum_{i_1,i_3\dots,i_{2n-1}\geq0\atop i_{2k-1}+i_{2k}=0}(-1)^{\ell(\si)+\sum_{k=1}^ni_{2k-1}}
2^{n-\sum_{k=1}^n\de_{i_{2k-1},0}}\rb
Q_{\al_1+i_{\si^{-1}(1)}}\cdots Q_{\al_{2n}+i_{\si^{-1}(2n)}}.
\end{equation}
In particular $\sq_{(m, n)}=Q_mQ_n+2\sum_{i=1}^{n}(-1)^iQ_{m+i}Q_{n-i}$.
\end{prop}

Next we give the following key relations for NSQF.
\begin{prop}
For any $\al\in\mc$ and $n\geq2$,
\beq\label{re}\sum_{i=1}^{n-1}\sq_{\al,i,n-i}=0.\eeq
\end{prop}
\begin{proof}
By the definition of $\sq_\al$, we only need to show that
\beq\label{ns}\sum_{i=1}^{n-1}\sq_{i,n-i}=0,\,n\geq2,\eeq
as $\sq_{\al,i,n-i}=\bY_{-\al}(\sq_{i,n-i})$.

Indeed using the recursive formula (\ref{rec}) (or by Prop. \ref{P:iter}), we have
\[\begin{split}
\sum_{i=1}^{n-1}\sq_{i,n-i}&=\sum_{i=1}^{n-1}
\lb Q_{i,n-i}+2\sum_{j=1}^{n-i}(-1)^jQ_{i+j,n-i-j}\rb
=\sum_{i=1}^{n-1}
Q_{i,n-i}+2\sum_{i=1}^{n-1}\sum_{k=i+1}^n(-1)^{k-i}Q_{k,n-k}\\
&=\sum_{i=1}^{n-1}
Q_{i,n-i}+2\sum_{k=2}^n\lb\sum_{i=1}^{k-1}(-1)^{k-i}\rb Q_{k,n-k}
=\sum_{k=0}^n(-1)^{k-1}Q_{k,n-k}=0.
\end{split}\]
\end{proof}
From the proof above, we know that relation (\ref{ns}) is equivalent to the Euler relation (\ref{np}). In particular, we know that those $\sq_\al$'s, with $\al$ ranging over strict compositions, are not linearly independent by letting $n$ to be odd in (\ref{ns}). Instead, we will find another index set for a natural basis of $\mP$ from NSQF later.

\subsection{The right-Pieri rule for NSQF}
Next we relate the operators $\bY_{-m}\,(m\in\bZ)$ with those $Q_s\,(s\geq1)$ to obtain a right-Pieri rule for NSQF.
\begin{lem}
For $s\geq1$ and $m\in\bZ$, and for $F\in\mP$,
\beq\label{yq}\bY_{-m}(F)Q_s=\bY_{-m}(FQ_s)+2\sum_{i=0}^{s-1}\bY_{-m-s+i}(FQ_i).\eeq
\end{lem}
\begin{proof}
We first prove that
\[\bY_{-m}(F)Q_s=\bY_{-m}(FQ_s)+2\sum_{j\geq1}(-1)^{j-1}\bY_{-m-j}(F)Q_{s-j}.\]
By definition,
\[\begin{split}
\bY_{-m}(FQ_s)&=\sum_{i\geq0}(-1)^iQ_{m+i}K_{\emp_i}^\perp(FQ_s)
=\sum_{i\geq0}(-1)^iQ_{m+i}\sum_{j=0}^iK_{\emp_{i-j}}^\perp(F)
K_{\emp_j}^\perp(Q_s)\\
&=\sum_{j\geq0}\sum_{k\geq 0}(-1)^{j+k}Q_{m+j+k}K_{\emp_k}^\perp(F)
K_{\emp_j}^\perp(Q_s)\\
&=\sum_{j\geq0}(-1)^j\lb\sum_{k\geq 0}(-1)^kQ_{m+j+k}K_{\emp_k}^\perp(F)\rb
K_{\emp_j}^\perp(Q_s)\\
&=\bY_{-m}(F)Q_s+2\sum_{j\geq1}(-1)^j\bY_{-m-j}(F)Q_{s-j},
\end{split}\]
where we use Lemma \ref{kq} to get the equalities.

Now we prove (\ref{yq}) by induction on $s$: the above formula clarifies the case when $s=1$. Furthermore,
\[\begin{split}
\bY_{-m}(F)Q_s&=\bY_{-m}(FQ_s)+2\sum_{j=1}^s(-1)^{j-1}\bY_{-m-j}(F)Q_{s-j}\\
&=\bY_{-m}(FQ_s)+2\sum_{j=1}^s(-1)^{j-1}
\lb\bY_{-m-j}(FQ_{s-j})+2\sum_{k=0}^{s-j-1}\bY_{-m-s+k}(FQ_k)\rb\\
&=\bY_{-m}(FQ_s)+2\sum_{j=0}^{s-1}(-1)^{s-1-j}
\bY_{-m-s+j}(FQ_j)+4\sum_{k=0}^{s-2}
\lb\sum_{j=1}^{s-k-1}(-1)^{j-1}\rb\bY_{-m-s+k}(FQ_k)\\
&=\bY_{-m}(FQ_s)+2\sum_{j=0}^{s-1}(-1)^{s-1-j}\bY_{-m-s+j}(FQ_j)
+2\sum_{k=0}^{s-2}\lb1-(-1)^{s-k-1}\rb\bY_{-m-s+k}(FQ_k)\\
&=\bY_{-m}(FQ_s)+2\sum_{j=0}^{s-1}\bY_{-m-s+j}(FQ_j).
\end{split}
\]
\end{proof}

In order to describe our Pieri rule, we need to introduce the following notation: for any compositions $\al,\be$, we write $\al\subset_s\be,\,s\in\bN$ if

(1)\quad $|\be|=|\al|+s$,

(2)\quad $\al_i\leq\be_i$ for all $1\leq i\leq \ell(\al)$,

(3)\quad $\ell(\be)\leq\ell(\al)+1$.

\medskip
By the formula (\ref{yq}) and the definition of $\sq_\al$'s, it is easy to deduce that
\begin{theorem}
For any composition $\al$, $\sq_\al$ satisfies the following right Pieri rule for multiplication by $Q_s$:
\beq\label{pie}\sq_\al Q_s=\sum_{\al\subset_s\be}2^{\ell(\be'-\al)}\sq_\be,\eeq
where $\be'=(\be_1,\dots,\be_{\ell(\al)})$ such that $\be'-\al\in\bN_0^{\ell(\al)}$.
\end{theorem}
\begin{proof}
We prove the formula by induction on $\ell(\al)$. If $\ell(\al)=1$, it can be obtained by letting $F=1$ in (\ref{yq}). If $\ell(\al)>1$, set
$\sq_\al=\bY_{-m}(\sq_{\ga})$, then by (\ref{yq}) and induction hypothesis,
we have
\[\begin{split}
\sq_\al Q_s&=\bY_{-m}(\sq_\ga Q_s)
+2\sum_{i=0}^{s-1}\bY_{-m-s+i}(\sq_\ga Q_i)\\
&=\bY_{-m}\lb\sum_{\ga\subset_s\eta}2^{\ell(\eta'-\ga)}\sq_\eta\rb
+2\sum_{i=0}^{s-1}\bY_{-m-s+i}\lb\sum_{\ga\subset_i\eta}2^{\ell(\eta'-\ga)}\sq_\eta\rb\\
&=\sum_{\ga\subset_s\eta}2^{\ell(\eta'-\ga)}\sq_{m,\,\eta}
+\sum_{i=0}^{s-1}\sum_{\ga\subset_i\eta}2^{\ell(\eta'-\ga)+1}\sq_{m+s-i,\,\eta}
=\sum_{\al\subset_s\be}2^{\ell(\be'-\al)}\sq_\be.
\end{split}\]
\end{proof}

%The disadvantage for such Pieri rule comes from the linear dependency of $\sq_\be$'s. Hence, the first crucial problem for NSQF is to extract a natural basis from the $\sq_\al$'s, which seems very complicated.

\begin{rem}
Applying the forgetful map $\pi$ to (\ref{pie}), one gets the Pieri rule of Schur's Q-functions \cite[Ch. III, (8.15)]{Mac}:
\beq\label{cpi}\mS_\mu q_s=\sum_\la 2^{b(\la/\mu)}\mS_\la\eeq
for any strict partition $\mu$, where the sum is over all strict $\la\supset\mu$ such that $\la/\mu$ is an $s$-horizontal strip, and $b(\la/\mu)$ is the number of $i\geq1$ such that $\la/\mu$ has a box in the $(i+1)$th column but not in the $i$th one.

Indeed, the special case for $\al$ with one part can illustrate the point. Assuming that $\al=(r),\,r\in\bN$, we need to consider two kinds of $\be$ with $\al\subset_s\be$:

(1) $(r+s)$;\quad (2) $(r+i,s-i),\,i=0,\dots,s-1$.

Case (1) gives the right coefficient 2. For (2) when $s-i<r$ gives the right coefficient $2-\de_{i,0}$. Now for (2) with $s-i\geq r$, we have the pairs
$(r+i,s-i)$ and $(s-i,r+i)$. When $i=0$ and $s>r$, then $\sq_{(r,s)}$ has coefficient $1$ while $\sq_{(s,r)}$ has $2$, thus gives $\mS_{(s,r)}$ the right $1$. When $i>0$ and $s-i>r$, then both $\sq_{(r+i,s-i)}$ and $\sq_{(s-i,r+i)}$ have coefficient $2$, they cancel with each other when projecting to the commutative case.

\end{rem}

\begin{cor}
For any composition $\al=(\al_1,\dots,\al_r)$, $Q_\al$ has the following expansion in terms of the NSQF:
\beq\label{exp} Q_\al=\sum_{\be^{(1)}=(\al_1)\subset_{\al_2}
\cdots\subset_{\al_r}\be^{(r)}=\be}
2^{\sum_{i=2}^r\ell({\be^{(i)}}'-\be^{(i-1)})}\sq_\be,\eeq
where ${\be^{(i)}}'=(\be^{(i)}_1,\dots,\be^{(i)}_{\ell(\be^{(i-1)})})$, a truncation of $\be^{(i)}$.
\end{cor}

\section{A natural basis in the peak algebra from NSQF}

By the expansion (\ref{exp}), we know that NSQF's also linearly span the peak subalgebra $\mP$. In this section we extract a natural basis for $\mP$ from the NSQF.

First we recall that Schur's Q-functions satisfy the anti-symmetric relation:
\[\mS_{\al}=-\mS_{\al_{ij}}\]
for any $i<j$, where $\al_{ij}$ is the composition derived from $\al$ by exchanging its $i$th and $j$th parts. Naturally, $\mS_{\al}$'s with $\al$ ranging over the strict partitions form a linear basis of $\Om$. On the other hand, since the NSQF serve as the noncommutative lift of Schur's Q-functions, we know that the kernel of $\pi|_\mP$ is spanned by those $\sq_\al+\sq_{\al_{ij}}$.

Now in order to define a natural basis from the NSQF, we need the following notion.

\begin{defn}
All the compositions of $n$ with peak subsets of $[n]$ as descent sets are those with all parts bigger than 1 except maybe only the last part equal to $1$. We call them the \textit{peak compositions}. We denote by $\pp\mc_n$ the set of peak compositions of $n$ and write $\pp\mc=\bigcup\limits^._{n\geq1}\pp\mc_n$.
\end{defn}
For example,
\[\pp_5=\{\emp_5,\,\{4\},\,\{3\},\,\{2\},\,\{2,4\}\}\]
and correspondingly
\[\pp\mc_5=\{5,\,41,\,32,\,23,\,2^21\}.\]

\begin{prop}
For  $\al=(\al_1,\dots,\al_r)\in\pp\mc$, the right-Pieri rule
(\ref{pie})
can be reformulated as
\beq\label{pie'}
\sq_\al Q_s=\sum_{\be\in\pp\mc\atop\al\subset_s\be}
2^{\ell(\be'-\al)-\de_{\al_r,1}(1-\de_{\be_{r+1},0})}\sq_\be,\eeq
\end{prop}

\begin{proof}
For the right-Pieri rule (\ref{pie}), if $\al\in\pp\mc$, then there exist two situations for the composition $\be$ with $\al\subset_s\be$:

\noindent(1)\quad $\be$ is still a peak composition.

\noindent(2)\quad $\be$ has the second last part equal to 1.

For case (1), nothing needs to be modified. For case (2), we can replace such $\sq_\be$ by a linear combination of NSQF's indexed by peak compositions appearing in case (1), since
\[\sq_\be=\sq_{\ga,1,m}=-\sum_{i=2}^m\sq_{\ga,i,m+1-i}\]
by relation (\ref{re}). That gives the formula we desired.
\end{proof}

\begin{theorem}
For $n\in\bN$, those $\sq_\al$'s with $\al\in\pp\mc_n$ form a linear basis of $\mP_n$.
\end{theorem}
\begin{proof}
Since $|\pp\mc_n|=|\pp_n|=f_{n-1}$, the $(n-1)$th Fibonacci number, we only need to show that any $Q_\be,\,\be\vDash n$, can be spanned by these  $\sq_\al$'s with $\al\in\pp\mc_n$, which can be seen by repeatedly applying the modified right-Pieri rule (\ref{pie'}).
\end{proof}

\begin{exam}
For those small $n$, we write down the natural basis and $2^{n-1}-f_{n-1}$ complete linear relations for $\sq_\al$'s, $\al\vDash n$. We also underline those elements derived from the natural basis by the complete relations. The computation for these relations heavily relies on the key relation (\ref{re}) and the right-Pieri rule for the NSQF.

\medskip\noindent
\begin{tabular}{|c|c|c|}
\hline
$n$& basis & relations \\
\hline
1& $\sq_1$ & $\backslash$\\
\hline
2& $\sq_2$ & $\sq_{1^2}$\\
\hline
3& $ \sq_3,\,\sq_{21}$ & $\sq_{1^3},\,\sq_{21}+\ull{\sq_{12}}$\\
\hline
4& $ \sq_4,\,\sq_{31},\,\sq_{2^2}$ & $\sq_{1^4},\,\sq_{21^2},\,\sq_{121}+\ull{\sq_{1^22}},\,
\ull{\sq_{121}}+2\sq_{2^2},
\sq_{31}+\sq_{2^2}+\ull{\sq_{13}}$\\
\hline
5& $\sq_5,\,\sq_{41},\,\sq_{32},\,\sq_{23},\,\sq_{2^21}$ & $\sq_{1^5},\,\sq_{21^3},\,\sq_{31^2},\,\sq_{121^2},\,
\sq_{1^221}+\ull{\sq_{1^32}},\,
\sq_{2^21}+\ull{\sq_{212}},$\\
&&$\sq_{2^21}+\ull{\sq_{12^2}},\,2\sq_{12^2}+\ull{\sq_{1^221}},\,
2(\sq_{32}+\sq_{23})+\sq_{2^21}+\ull{\sq_{131}},$\\
&&$\sq_{131}+\sq_{12^2}+\ull{\sq_{1^23}},\,
\sq_{41}+\sq_{32}+\sq_{23}+\ull{\sq_{14}}$\\
\hline
6& $\sq_6,\,\sq_{51},\,\sq_{42},\,\sq_{24},\,\sq_{3^2},$ & $\sq_{1^6},\,\sq_{21^4},\,\sq_{31^3},\,\sq_{41^2},\,
\sq_{131^2},\,\sq_{2^21^2},\,\sq_{121^3},\,\sq_{1^221^2},$\\
&$\sq_{321},\,\sq_{231},\sq_{2^3}$&
$\sq_{321}+\ull{\sq_{312}},\,\sq_{2121}+\ull{\sq_{21^22}},\,
\sq_{12^21}+\ull{\sq_{1212}},\,\sq_{1^321}+\ull{\sq_{1^42}},$\\
&&$\sq_{12^21}+\ull{\sq_{1^22^2}},\,
\ull{\sq_{1^321}}+2\sq_{1^22^2},\,\ull{\sq_{2121}}+2\sq_{2^3},$\\
&&
$2(\sq_{132}+\sq_{123})+\sq_{12^21}+\ull{\sq_{1^231}},$\\
&&
$2(\sq_{231}+\sq_{3^2})+\sq_{321}+\ull{\sq_{132}}+\sq_{2^3},$\\
&&
$2(\sq_{42}+\sq_{24})+\sq_{141}+\sq_{321}+
\ull{\sq_{123}}+\sq_{2^3},$\\
&&
$2(\sq_{42}+\sq_{24}+\sq_{3^2})+\ull{\sq_{141}}+\sq_{321}+
\sq_{231},$\\
&&
$\ull{\sq_{12^21}}+2(\sq_{321}+
\sq_{231}+\sq_{132}+\sq_{123})+4\sq_{2^3},$\\
&&
$\sq_{231}+\sq_{2^3}+\ull{\sq_{213}},\,
\sq_{1^231}+\sq_{1^22^2}+\ull{\sq_{1^33}},$\\
&&$\sq_{141}+\sq_{132}+\sq_{123}+\ull{\sq_{1^24}},\,
\sq_{51}+\sq_{42}+\sq_{3^2}+\sq_{24}+\ull{\sq_{15}}$\\
\hline
\end{tabular}

\end{exam}

\subsection{Combinatorial description for NSQF}
In order to convince the reader that our basis from the NSQF is quite natural, we give some interesting combinatorial description here. First let us introduce the following combinatorial objects.

In \cite{BBSSZ}, the authors defined the notion of immaculate tableaux to discuss the relations between the classical bases of the NSym and their immaculate basis. Let $\al,\,\be$ be compositions. Recall that an \textit{immaculate tableau} of shape $\al$ and content $\be$ is a labeling of the boxes of the diagram of $\al$ by positive integers such that:

\noindent(a)\quad the number of boxes labeled by $i$ is $\be_i$;

\noindent(b)\quad the sequence of entries in each row, from left to right, is weakly increasing;

\noindent(c)\quad the sequence of entries in the first column, from top to bottom, is increasing.

\noindent An immaculate tableau is said to be \textit{standard} if it has content $1^{|\al|}$.
%Let $K_{\al,\be}$ denote the number of immaculate tableaux of shape $\al$ and content $\be$.
Correspondingly, they created a labeled poset $\fc$ on $\mc$, called the \textit{immaculate poset}, such that $\be$ covers $\al$ if $\al\subset_1\be$. Write $\be\stackrel{m}{\rw}\al$ if $\be$ is obtained by adding 1 to the $m$th part of $\al$. Note that standard immaculate tableaux of shape $\be/\al$ one-to-one correspond to maximal chains on this poset from $\al$ to $\be$.

Now we just need to focus on the subposet of $\fc$ on $\pp\mc$ and call it the \textit{peak composition poset}, denoted by $\fp\fc$. The first few levels of $\fp\fc$ are portrayed as follows.
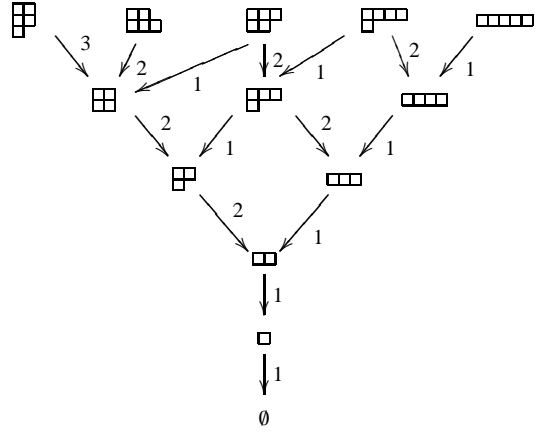
\[\xy 0;/r.5pc/:
{\ar^{\stt{1}}(0,-3.5)*{};(0,-1)*{}**\dir{-};};
{\ar^{\stt{1}}(0,1.5)*{};(0,4)*{}**\dir{-};};
{\ar^{\stt{2}}(-1,5.5)*{};(-4,9)*{}**\dir{-};};
{\ar^{\stt{1}}(1,5.5)*{};(4,9)*{}**\dir{-};};
{\ar^{\stt{2}}(-6,11.5)*{};(-8,14)*{}**\dir{-};};
{\ar^{\stt{1}}(6,11.5)*{};(8,14)*{}**\dir{-};};
{\ar^{\stt{1}}(-4,11.5)*{};(-2,14)*{}**\dir{-};};
{\ar^{\stt{2}}(4,11.5)*{};(2,14)*{}**\dir{-};};
{\ar^{\stt{3}}(-11,16.5)*{};(-13,19)*{}**\dir{-};};
{\ar^{\stt{1}}(11,16.5)*{};(13,19)*{}**\dir{-};};
%{\ar^{\stt{2}}(-1,16.5)*{};(-5,19)*{}**\dir{-};};
{\ar^{\stt{1}}(1,16.5)*{};(5,19)*{}**\dir{-};};
{\ar^{\stt{2}}(0,16.5)*{};(0,18.5)*{}**\dir{-};};
{\ar^{\stt{1}}(-8,15.5)*{};(-1,18.5)*{}**\dir{-};};
{\ar^{\stt{2}}(-9,16.5)*{};(-8,18.5)*{}**\dir{-};};
{\ar^{\stt{2}}(9,16.5)*{};(8,19)*{};};
(0,-5)*{\stt{\emptyset}};
(0,0)*{\xy 0;/r.07pc/:
(0,0)*{};(5,0)*{}**\dir{-};
(0,-5)*{};(5,-5)*{}**\dir{-};
(0,0)*{};(0,-5)*{}**\dir{-};
(5,0)*{};(5,-5)*{}**\dir{-};
\endxy};
(0,5)*{\xy 0;/r.07pc/:
(0,0)*{};(10,0)*{}**\dir{-};
(0,-5)*{};(10,-5)*{}**\dir{-};
(0,0)*{};(0,-5)*{}**\dir{-};
(5,0)*{};(5,-5)*{}**\dir{-};
(10,0)*{};(10,-5)*{}**\dir{-};
\endxy};
(-5,10)*{\xy 0;/r.07pc/:
(0,0)*{};(10,0)*{}**\dir{-};
(0,-5)*{};(10,-5)*{}**\dir{-};
(0,-10)*{};(5,-10)*{}**\dir{-};
(0,0)*{};(0,-10)*{}**\dir{-};
(5,0)*{};(5,-10)*{}**\dir{-};
(10,0)*{};(10,-5)*{}**\dir{-};
\endxy};
(5,10)*{\xy 0;/r.07pc/:
(0,0)*{};(15,0)*{}**\dir{-};
(0,-5)*{};(15,-5)*{}**\dir{-};
(0,0)*{};(0,-5)*{}**\dir{-};
(5,0)*{};(5,-5)*{}**\dir{-};
(10,0)*{};(10,-5)*{}**\dir{-};
(15,0)*{};(15,-5)*{}**\dir{-};
\endxy};
(-10,15)*{\xy 0;/r.07pc/:
(0,0)*{};(10,0)*{}**\dir{-};
(0,-5)*{};(10,-5)*{}**\dir{-};
(0,-10)*{};(10,-10)*{}**\dir{-};
(0,0)*{};(0,-10)*{}**\dir{-};
(5,0)*{};(5,-10)*{}**\dir{-};
(10,0)*{};(10,-10)*{}**\dir{-};
\endxy};
(0,15)*{\xy 0;/r.07pc/:
(0,0)*{};(15,0)*{}**\dir{-};
(0,-5)*{};(15,-5)*{}**\dir{-};
(0,-10)*{};(5,-10)*{}**\dir{-};
(0,0)*{};(0,-10)*{}**\dir{-};
(5,0)*{};(5,-10)*{}**\dir{-};
(10,0)*{};(10,-5)*{}**\dir{-};
(15,0)*{};(15,-5)*{}**\dir{-};
\endxy};
(10,15)*{\xy 0;/r.07pc/:
(0,0)*{};(20,0)*{}**\dir{-};
(0,-5)*{};(20,-5)*{}**\dir{-};
(0,0)*{};(0,-5)*{}**\dir{-};
(5,0)*{};(5,-5)*{}**\dir{-};
(10,0)*{};(10,-5)*{}**\dir{-};
(15,0)*{};(15,-5)*{}**\dir{-};
(20,0)*{};(20,-5)*{}**\dir{-};
\endxy};
(-15,20)*{\xy 0;/r.07pc/:
(0,0)*{};(10,0)*{}**\dir{-};
(0,-5)*{};(10,-5)*{}**\dir{-};
(0,-10)*{};(10,-10)*{}**\dir{-};
(0,-15)*{};(5,-15)*{}**\dir{-};
(0,0)*{};(0,-15)*{}**\dir{-};
(5,0)*{};(5,-15)*{}**\dir{-};
(10,0)*{};(10,-10)*{}**\dir{-};
\endxy};
(-7.5,20)*{\xy 0;/r.07pc/:
(0,0)*{};(10,0)*{}**\dir{-};
(0,-5)*{};(15,-5)*{}**\dir{-};
(0,-10)*{};(15,-10)*{}**\dir{-};
(0,0)*{};(0,-10)*{}**\dir{-};
(5,0)*{};(5,-10)*{}**\dir{-};
(10,0)*{};(10,-10)*{}**\dir{-};
(15,-5)*{};(15,-10)*{}**\dir{-};
\endxy};
(0,20)*{\xy 0;/r.07pc/:
(0,0)*{};(15,0)*{}**\dir{-};
(0,-5)*{};(15,-5)*{}**\dir{-};
(0,-10)*{};(10,-10)*{}**\dir{-};
(0,0)*{};(0,-10)*{}**\dir{-};
(5,0)*{};(5,-10)*{}**\dir{-};
(10,0)*{};(10,-10)*{}**\dir{-};
(15,0)*{};(15,-5)*{}**\dir{-};
\endxy};
(7.5,20)*{\xy 0;/r.07pc/:
(0,0)*{};(20,0)*{}**\dir{-};
(0,-5)*{};(20,-5)*{}**\dir{-};
(0,-10)*{};(5,-10)*{}**\dir{-};
(0,0)*{};(0,-10)*{}**\dir{-};
(5,0)*{};(5,-10)*{}**\dir{-};
(10,0)*{};(10,-5)*{}**\dir{-};
(15,0)*{};(15,-5)*{}**\dir{-};
(20,0)*{};(20,-5)*{}**\dir{-};
\endxy};
(15,20)*{\xy 0;/r.07pc/:
(0,0)*{};(25,0)*{}**\dir{-};
(0,-5)*{};(25,-5)*{}**\dir{-};
(0,0)*{};(0,-5)*{}**\dir{-};
(5,0)*{};(5,-5)*{}**\dir{-};
(10,0)*{};(10,-5)*{}**\dir{-};
(15,0)*{};(15,-5)*{}**\dir{-};
(20,0)*{};(20,-5)*{}**\dir{-};
(25,0)*{};(25,-5)*{}**\dir{-};
\endxy};\endxy\]

\begin{defn} A tableau of shape $\al$ and content $\be$ is called a \textit{peak composition tableau} (PCT)
if (1) it is an immaculate tableau, and (2) for any $i:1\leq i\leq r=\ell(\be)$, the subdiagram of $\al$ with labeling in $\{1,\dots,i\}$ gives a peak composition.
\end{defn}
The definition clearly implies that $\al\in\pp\mc$. Let $\mb{PCT}(\al,\be)$ denote the set of peak composition tableaux of shape $\al$ and content $\be$, and let $\mb{p}_{\al,\be}$ be its cardinality. Then one can easily see that

%An immaculate tableau of shape $\al$ and content $\be$ is called a \textit{peak composition tableau} (abbreviated as PCT) if it satisfies the following extra condition:
%
%\textit{Let $r=\ell(\be)$. For any $i:1\leq i\leq r$, the subdiagram of $\al$ with labeling in $\{1,\dots,i\}$ gives a peak composition. In particular, necessarily $\al\in\pp\mc$.}
%\end{defn}

\noindent (1)\quad $\mb{p}_{\al,\al}=1$ for any $\al\in\pp\mc$. Such unique tableau consists of $\al_i$ many $i$'s in the $i$th row.

\noindent (2)\quad $\mb{p}_{\al,\be}=0$ unless $\be\leq_l\al$, where $\leq_l$ represents the lexicographic order on compositions.

For example,
\[\mb{PCT}(342,2214)=\left\{\,
\raisebox{1.5em}{\xy 0;/r.16pc/:
(0,0)*{};(15,0)*{}**\dir{-};
(0,-5)*{};(20,-5)*{}**\dir{-};
(0,-10)*{};(20,-10)*{}**\dir{-};
(0,-15)*{};(10,-15)*{}**\dir{-};
(0,0)*{};(0,-15)*{}**\dir{-};
(5,0)*{};(5,-15)*{}**\dir{-};
(10,0)*{};(10,-15)*{}**\dir{-};
(15,0)*{};(15,-10)*{}**\dir{-};
(20,-5)*{};(20,-10)*{}**\dir{-};
(2.5,-2.5)*{\stt{1}};(7.5,-2.5)*{\stt{1}};(12.5,-2.5)*{\stt{2}};
(2.5,-7.5)*{\stt{2}};(7.5,-7.5)*{\stt{3}};(12.5,-7.5)*{\stt{4}};
(17.5,-7.5)*{\stt{4}};
(2.5,-12.5)*{\stt{4}};(7.5,-12.5)*{\stt{4}};
\endxy}\,,\,\raisebox{1.5em}{\xy 0;/r.16pc/:
(0,0)*{};(15,0)*{}**\dir{-};
(0,-5)*{};(20,-5)*{}**\dir{-};
(0,-10)*{};(20,-10)*{}**\dir{-};
(0,-15)*{};(10,-15)*{}**\dir{-};
(0,0)*{};(0,-15)*{}**\dir{-};
(5,0)*{};(5,-15)*{}**\dir{-};
(10,0)*{};(10,-15)*{}**\dir{-};
(15,0)*{};(15,-10)*{}**\dir{-};
(20,-5)*{};(20,-10)*{}**\dir{-};
(2.5,-2.5)*{\stt{1}};(7.5,-2.5)*{\stt{1}};(12.5,-2.5)*{\stt{3}};
(2.5,-7.5)*{\stt{2}};(7.5,-7.5)*{\stt{2}};(12.5,-7.5)*{\stt{4}};
(17.5,-7.5)*{\stt{4}};
(2.5,-12.5)*{\stt{4}};(7.5,-12.5)*{\stt{4}};
\endxy}\,,
\,\raisebox{1.5em}{\xy 0;/r.16pc/:
(0,0)*{};(15,0)*{}**\dir{-};
(0,-5)*{};(20,-5)*{}**\dir{-};
(0,-10)*{};(20,-10)*{}**\dir{-};
(0,-15)*{};(10,-15)*{}**\dir{-};
(0,0)*{};(0,-15)*{}**\dir{-};
(5,0)*{};(5,-15)*{}**\dir{-};
(10,0)*{};(10,-15)*{}**\dir{-};
(15,0)*{};(15,-10)*{}**\dir{-};
(20,-5)*{};(20,-10)*{}**\dir{-};
(2.5,-2.5)*{\stt{1}};(7.5,-2.5)*{\stt{1}};(12.5,-2.5)*{\stt{4}};
(2.5,-7.5)*{\stt{2}};(7.5,-7.5)*{\stt{2}};(12.5,-7.5)*{\stt{3}};
(17.5,-7.5)*{\stt{4}};
(2.5,-12.5)*{\stt{4}};(7.5,-12.5)*{\stt{4}};
\endxy}\,,
\,\raisebox{1.5em}{\xy 0;/r.16pc/:
(0,0)*{};(15,0)*{}**\dir{-};
(0,-5)*{};(20,-5)*{}**\dir{-};
(0,-10)*{};(20,-10)*{}**\dir{-};
(0,-15)*{};(10,-15)*{}**\dir{-};
(0,0)*{};(0,-15)*{}**\dir{-};
(5,0)*{};(5,-15)*{}**\dir{-};
(10,0)*{};(10,-15)*{}**\dir{-};
(15,0)*{};(15,-10)*{}**\dir{-};
(20,-5)*{};(20,-10)*{}**\dir{-};
(2.5,-2.5)*{\stt{1}};(7.5,-2.5)*{\stt{1}};(12.5,-2.5)*{\stt{4}};
(2.5,-7.5)*{\stt{2}};(7.5,-7.5)*{\stt{2}};(12.5,-7.5)*{\stt{4}};
(17.5,-7.5)*{\stt{4}};
(2.5,-12.5)*{\stt{3}};(7.5,-12.5)*{\stt{4}};
\endxy}\,\right\}.
\]
By contrast,
\[T=\raisebox{1.5em}{\xy 0;/r.16pc/:
(0,0)*{};(15,0)*{}**\dir{-};
(0,-5)*{};(20,-5)*{}**\dir{-};
(0,-10)*{};(20,-10)*{}**\dir{-};
(0,-15)*{};(10,-15)*{}**\dir{-};
(0,0)*{};(0,-15)*{}**\dir{-};
(5,0)*{};(5,-15)*{}**\dir{-};
(10,0)*{};(10,-15)*{}**\dir{-};
(15,0)*{};(15,-10)*{}**\dir{-};
(20,-5)*{};(20,-10)*{}**\dir{-};
(2.5,-2.5)*{\stt{1}};(7.5,-2.5)*{\stt{1}};(12.5,-2.5)*{\stt{2}};
(2.5,-7.5)*{\stt{2}};(7.5,-7.5)*{\stt{4}};(12.5,-7.5)*{\stt{4}};
(17.5,-7.5)*{\stt{4}};
(2.5,-12.5)*{\stt{3}};(7.5,-12.5)*{\stt{4}};
\endxy}\]
is only an immaculate tableau, but not a peak composition tableau.

Note that any standard peak composition tableau is just a standard immaculate tableau with a peak composition as its shape, but the converse is \textit{not} true.
Hence, standard peak composition tableaux of shape $\be/\al$, $\al,\,\be\in\pp\mc$, one-to-one correspond to maximal chains on $\fp\fc$ from $\al$ to $\be$.

Given $T\in\mb{PCT}$, we also define $p_i(T),\,i\geq1$ as the number of \textit{distinct} integers appearing in the $i$th row of $T$ minus one. Let $p(T)=\sum_{i\geq1}p_i(T)$ and $m(T)$ be the number of boxes in the leftmost column whose right and bottom adjacent boxes are labeled with the same integer. For example,
\[T=\raisebox{2.4em}{\xy 0;/r.16pc/:
(0,0)*{};(15,0)*{}**\dir{-};
(0,-5)*{};(20,-5)*{}**\dir{-};
(0,-10)*{};(20,-10)*{}**\dir{-};
(0,-15)*{};(15,-15)*{}**\dir{-};
(0,-20)*{};(15,-20)*{}**\dir{-};
(0,-25)*{};(5,-25)*{}**\dir{-};
(0,0)*{};(0,-25)*{}**\dir{-};
(5,0)*{};(5,-25)*{}**\dir{-};
(10,0)*{};(10,-20)*{}**\dir{-};
(15,0)*{};(15,-10)*{}**\dir{-};
(15,-15)*{};(15,-20)*{}**\dir{-};
(20,-5)*{};(20,-10)*{}**\dir{-};
(2.5,-2.5)*{\stt{1}};(7.5,-2.5)*{\stt{1}};(12.5,-2.5)*{\stt{2}};
(2.5,-7.5)*{\stt{2}};(7.5,-7.5)*{\stt{3}};(12.5,-7.5)*{\stt{3}};
(17.5,-7.5)*{\stt{5}};
(2.5,-12.5)*{\stt{3}};(7.5,-12.5)*{\stt{4}};
(2.5,-17.5)*{\stt{4}};(7.5,-17.5)*{\stt{4}};
(12.5,-17.5)*{\stt{5}};(2.5,-22.5)*{\stt{5}};
\endxy}\]
Then $p_1(T)=p_3(T)=p_4(T)=1,\,p_2(T)=2,\,p_5(T)=0$, thus $p(T)=5$. Also, $m(T)=2$.

Now using the reformulated Pieri rule (\ref{pie'}) and the terminology of PCT, we can simplify formula (\ref{exp}) as follows:
\begin{theorem}
For any composition $\al$, $Q_\al$ has the following expansion in terms of the natural basis from NSQF:
\beq\label{exp'} Q_\al=\sum_{\be\in\pp\mc\atop\al\leq_l\be}\lb\sum_{T\in\bs{PCT}(\be,\al)}
2^{p(T)-m(T)}\rb\sq_\be.\eeq
\end{theorem}

Meanwhile, we also find that
\begin{lem}\label{da}
For any $T\in\mb{PCT}(\be,\al)$, $p(T)-m(T)+\ell(\be)-\ell(\al)\geq0$.
\end{lem}
\begin{proof} By definition, $p(T)+\ell(\be)$ is the sum of numbers of distinct integers appearing in each row of $T$. If there exists a box in the leftmost column whose right and bottom adjacent boxes are both labeled with some integer $i$, then such $i$ should be counted at least twice. Hence, we have $p(T)+\ell(\be)\geq m(T)+\ell(\al)$.
\end{proof}

\begin{exam}
For $\al=2^3$, we have
\[\begin{split}
\mb{PCT}(\cdot,\al)=&\left\{
\,\raisebox{1.5em}{\xy 0;/r.16pc/:
(0,0)*{};(10,0)*{}**\dir{-};
(0,-5)*{};(10,-5)*{}**\dir{-};
(0,-10)*{};(10,-10)*{}**\dir{-};
(0,-15)*{};(10,-15)*{}**\dir{-};
(0,0)*{};(0,-15)*{}**\dir{-};
(5,0)*{};(5,-15)*{}**\dir{-};
(10,0)*{};(10,-15)*{}**\dir{-};
(2.5,-2.5)*{\stt{1}};(7.5,-2.5)*{\stt{1}};
(2.5,-7.5)*{\stt{2}};(7.5,-7.5)*{\stt{2}};
(2.5,-12.5)*{\stt{3}};(7.5,-12.5)*{\stt{3}};
\endxy}\,,
\,\raisebox{1.5em}{\xy 0;/r.16pc/:
(0,0)*{};(10,0)*{}**\dir{-};
(0,-5)*{};(15,-5)*{}**\dir{-};
(0,-10)*{};(15,-10)*{}**\dir{-};
(0,-15)*{};(5,-15)*{}**\dir{-};
(0,0)*{};(0,-15)*{}**\dir{-};
(5,0)*{};(5,-15)*{}**\dir{-};
(10,0)*{};(10,-10)*{}**\dir{-};
(15,-5)*{};(15,-10)*{}**\dir{-};
(2.5,-2.5)*{\stt{1}};(7.5,-2.5)*{\stt{1}};
(2.5,-7.5)*{\stt{2}};(7.5,-7.5)*{\stt{2}};
(12.5,-7.5)*{\stt{3}};
(2.5,-12.5)*{\stt{3}};
\endxy}\,,
\,\raisebox{1.5em}{\xy 0;/r.16pc/:
(0,0)*{};(10,0)*{}**\dir{-};
(0,-5)*{};(20,-5)*{}**\dir{-};
(0,-10)*{};(20,-10)*{}**\dir{-};
(0,0)*{};(0,-10)*{}**\dir{-};
(5,0)*{};(5,-10)*{}**\dir{-};
(10,0)*{};(10,-10)*{}**\dir{-};
(15,-5)*{};(15,-10)*{}**\dir{-};
(20,-5)*{};(20,-10)*{}**\dir{-};
(2.5,-2.5)*{\stt{1}};(7.5,-2.5)*{\stt{1}};
(2.5,-7.5)*{\stt{2}};(7.5,-7.5)*{\stt{2}};
(12.5,-7.5)*{\stt{3}};(17.5,-7.5)*{\stt{3}};
\endxy}\,,
\,\raisebox{1.5em}{\xy 0;/r.16pc/:
(0,0)*{};(15,0)*{}**\dir{-};
(0,-5)*{};(15,-5)*{}**\dir{-};
(0,-10)*{};(10,-10)*{}**\dir{-};
(0,-15)*{};(5,-15)*{}**\dir{-};
(0,0)*{};(0,-15)*{}**\dir{-};
(5,0)*{};(5,-15)*{}**\dir{-};
(10,0)*{};(10,-10)*{}**\dir{-};
(15,0)*{};(15,-5)*{}**\dir{-};
(2.5,-2.5)*{\stt{1}};(7.5,-2.5)*{\stt{1}};
(12.5,-2.5)*{\stt{2}};(2.5,-7.5)*{\stt{2}};
(7.5,-7.5)*{\stt{3}};
(2.5,-12.5)*{\stt{3}};
\endxy}\,,
\,\raisebox{1.5em}{\xy 0;/r.16pc/:
(0,0)*{};(15,0)*{}**\dir{-};
(0,-5)*{};(15,-5)*{}**\dir{-};
(0,-10)*{};(10,-10)*{}**\dir{-};
(0,-15)*{};(5,-15)*{}**\dir{-};
(0,0)*{};(0,-15)*{}**\dir{-};
(5,0)*{};(5,-15)*{}**\dir{-};
(10,0)*{};(10,-10)*{}**\dir{-};
(15,0)*{};(15,-5)*{}**\dir{-};
(2.5,-2.5)*{\stt{1}};(7.5,-2.5)*{\stt{1}};
(12.5,-2.5)*{\stt{3}};(2.5,-7.5)*{\stt{2}};
(7.5,-7.5)*{\stt{2}};
(2.5,-12.5)*{\stt{3}};
\endxy}\,,
\,\raisebox{1.5em}{\xy 0;/r.16pc/:
(0,0)*{};(15,0)*{}**\dir{-};
(0,-5)*{};(15,-5)*{}**\dir{-};
(0,-10)*{};(15,-10)*{}**\dir{-};
(0,0)*{};(0,-10)*{}**\dir{-};
(5,0)*{};(5,-10)*{}**\dir{-};
(10,0)*{};(10,-10)*{}**\dir{-};
(15,0)*{};(15,-10)*{}**\dir{-};
(2.5,-2.5)*{\stt{1}};(7.5,-2.5)*{\stt{1}};
(12.5,-2.5)*{\stt{2}};(2.5,-7.5)*{\stt{2}};
(7.5,-7.5)*{\stt{3}};(12.5,-7.5)*{\stt{3}};
\endxy}\,,
\,\raisebox{1.5em}{\xy 0;/r.16pc/:
(0,0)*{};(15,0)*{}**\dir{-};
(0,-5)*{};(15,-5)*{}**\dir{-};
(0,-10)*{};(15,-10)*{}**\dir{-};
(0,0)*{};(0,-10)*{}**\dir{-};
(5,0)*{};(5,-10)*{}**\dir{-};
(10,0)*{};(10,-10)*{}**\dir{-};
(15,0)*{};(15,-10)*{}**\dir{-};
(2.5,-2.5)*{\stt{1}};(7.5,-2.5)*{\stt{1}};
(12.5,-2.5)*{\stt{3}};(2.5,-7.5)*{\stt{2}};
(7.5,-7.5)*{\stt{2}};(12.5,-7.5)*{\stt{3}};
\endxy}\,,\right.\\
&
\left.\,\raisebox{1.5em}{\xy 0;/r.16pc/:
(0,0)*{};(20,0)*{}**\dir{-};
(0,-5)*{};(20,-5)*{}**\dir{-};
(0,-10)*{};(10,-10)*{}**\dir{-};
(0,0)*{};(0,-10)*{}**\dir{-};
(5,0)*{};(5,-10)*{}**\dir{-};
(10,0)*{};(10,-10)*{}**\dir{-};
(15,0)*{};(15,-5)*{}**\dir{-};
(20,0)*{};(20,-5)*{}**\dir{-};
(2.5,-2.5)*{\stt{1}};(7.5,-2.5)*{\stt{1}};
(12.5,-2.5)*{\stt{2}};(17.5,-2.5)*{\stt{2}};
(2.5,-7.5)*{\stt{3}};(7.5,-7.5)*{\stt{3}};
\endxy}\,,
\,\raisebox{1.5em}{\xy 0;/r.16pc/:
(0,0)*{};(20,0)*{}**\dir{-};
(0,-5)*{};(20,-5)*{}**\dir{-};
(0,-10)*{};(10,-10)*{}**\dir{-};
(0,0)*{};(0,-10)*{}**\dir{-};
(5,0)*{};(5,-10)*{}**\dir{-};
(10,0)*{};(10,-10)*{}**\dir{-};
(15,0)*{};(15,-5)*{}**\dir{-};
(20,0)*{};(20,-5)*{}**\dir{-};
(2.5,-2.5)*{\stt{1}};(7.5,-2.5)*{\stt{1}};
(12.5,-2.5)*{\stt{2}};(17.5,-2.5)*{\stt{3}};
(2.5,-7.5)*{\stt{2}};(7.5,-7.5)*{\stt{3}};
\endxy}\,,
\,\raisebox{1.5em}{\xy 0;/r.16pc/:
(0,0)*{};(20,0)*{}**\dir{-};
(0,-5)*{};(20,-5)*{}**\dir{-};
(0,-10)*{};(10,-10)*{}**\dir{-};
(0,0)*{};(0,-10)*{}**\dir{-};
(5,0)*{};(5,-10)*{}**\dir{-};
(10,0)*{};(10,-10)*{}**\dir{-};
(15,0)*{};(15,-5)*{}**\dir{-};
(20,0)*{};(20,-5)*{}**\dir{-};
(2.5,-2.5)*{\stt{1}};(7.5,-2.5)*{\stt{1}};
(12.5,-2.5)*{\stt{3}};(17.5,-2.5)*{\stt{3}};
(2.5,-7.5)*{\stt{2}};(7.5,-7.5)*{\stt{2}};
\endxy}\,,
\,\raisebox{1.5em}{\xy 0;/r.16pc/:
(0,0)*{};(25,0)*{}**\dir{-};
(0,-5)*{};(25,-5)*{}**\dir{-};
(0,-10)*{};(5,-10)*{}**\dir{-};
(0,0)*{};(0,-10)*{}**\dir{-};
(5,0)*{};(5,-10)*{}**\dir{-};
(10,0)*{};(10,-5)*{}**\dir{-};
(15,0)*{};(15,-5)*{}**\dir{-};
(20,0)*{};(20,-5)*{}**\dir{-};
(25,0)*{};(25,-5)*{}**\dir{-};
(2.5,-2.5)*{\stt{1}};(7.5,-2.5)*{\stt{1}};
(12.5,-2.5)*{\stt{2}};(17.5,-2.5)*{\stt{2}};
(22.5,-2.5)*{\stt{3}};(2.5,-7.5)*{\stt{3}};
\endxy}\,,
\,\raisebox{1.5em}{\xy 0;/r.16pc/:
(0,0)*{};(25,0)*{}**\dir{-};
(0,-5)*{};(25,-5)*{}**\dir{-};
(0,-10)*{};(5,-10)*{}**\dir{-};
(0,0)*{};(0,-10)*{}**\dir{-};
(5,0)*{};(5,-10)*{}**\dir{-};
(10,0)*{};(10,-5)*{}**\dir{-};
(15,0)*{};(15,-5)*{}**\dir{-};
(20,0)*{};(20,-5)*{}**\dir{-};
(25,0)*{};(25,-5)*{}**\dir{-};
(2.5,-2.5)*{\stt{1}};(7.5,-2.5)*{\stt{1}};
(12.5,-2.5)*{\stt{2}};(17.5,-2.5)*{\stt{3}};
(22.5,-2.5)*{\stt{3}};(2.5,-7.5)*{\stt{2}};
\endxy}\,,
\,\raisebox{1.5em}{\xy 0;/r.16pc/:
(0,0)*{};(30,0)*{}**\dir{-};
(0,-5)*{};(30,-5)*{}**\dir{-};
(0,0)*{};(0,-5)*{}**\dir{-};
(5,0)*{};(5,-5)*{}**\dir{-};
(10,0)*{};(10,-5)*{}**\dir{-};
(15,0)*{};(15,-5)*{}**\dir{-};
(20,0)*{};(20,-5)*{}**\dir{-};
(25,0)*{};(25,-5)*{}**\dir{-};
(30,0)*{};(30,-5)*{}**\dir{-};
(2.5,-2.5)*{\stt{1}};(7.5,-2.5)*{\stt{1}};
(12.5,-2.5)*{\stt{2}};(17.5,-2.5)*{\stt{2}};
(22.5,-2.5)*{\stt{3}};(27.5,-2.5)*{\stt{3}};
\endxy}\,
\right\}.
\end{split}
\]
Then $p(T)=0,1,1,2,1,2,2,1,3,1,2,2,2$ successively and all $m(T)=0$ except the fourth one equal to 1. Hence,
\[Q_{2^3}=\sq_{2^3}+2(\sq_{231}+\sq_{24})+
4(\sq_{321}+\sq_6)+8(\sq_{3^2}+\sq_{51})+12\sq_{42}.\]
\end{exam}

\begin{cor}
For $n\in\bN$, $\{Q_\al\}_{\al\in\pp\mc_n}$ is also a linear basis of $\mP_n$.
\end{cor}
\begin{proof}
Note that for any $\al\in\pp\mc_n$, the unique $T\in\mb{PCT}(\al,\al)$
satisfies $p(T)=m(T)=0$. In particular, by (\ref{exp'}) the transition matrix between
$\{Q_\al\}_{\al\in\pp\mc_n}$ and $\{\sq_\al\}_{\al\in\pp\mc_n}$
is upper unitriangular with respect to the lexicographic order, and thus the former is also a basis.
\end{proof}

\begin{rem}
Via the bijection between peak sets and peak compositions, we can take
$\pp\mc$ as the index set for bases of the peak algebra and let
\[\Pi_\al:=\Pi_{D(\al)},\,\al\in\pp\mc.\]
By \cite[Prop. 3.4.]{Sch}, we have
\beq\label{qp}Q_\al=2^{\ell(\al)}\sum_{\be\in\pp\mc\atop
D(\be)\subseteq D(\al)\cup(D(\al)+1)}\Pi_\be\eeq
for any composition $\al$. Obviously, for $\be\in\pp\mc$ such that $D(\be)\subseteq D(\al)\cup(D(\al)+1)$, we must have $\al\leq_l\be$. In particular, the transition matrix between $\{Q_\al\}_{\al\in\pp\mc_n}$ and $\{\Pi_\al\}_{\al\in\pp\mc_n}$
is also upper triangular.
% For example,
%\[
%Q_{23}=\sq_{23}+2(\sq_{32}+\sq_{41}+\sq_5)=4(\Pi_{23}+\Pi_{32}+\Pi_5).
%\]
\end{rem}

For small $n\in\bN$, we list the transition matrices between $\{Q_\al\}_{\al\in\pp\mc_n}$ and $\{\sq_\al\}_{\al\in\pp\mc_n}$ (resp. $\{\Pi_\al\}_{\al\in\pp\mc_n}$), denoted by $M_n(Q,\sq)$ (resp. $M_n(Q,\Pi)$).

\[\begin{array}{|c|c|c|c|}
 \hline
  \quad n \quad&\quad \mb{index}\quad&\quad M_n(Q,\sq)\quad & M_n(Q,\Pi)\\ \hline
  \quad 3 \quad&\quad 21,3 \quad&\quad \begin{matrix}
    1&2\\0&1
  \end{matrix}\quad&\quad \begin{matrix}
    4&4\\0&2
  \end{matrix}\quad\\ \hline
  \quad 4 \quad&\quad 2^2,31,4 \quad&\quad \begin{matrix}
    1&2&2\\0&1&2\\0&0&1
  \end{matrix} \quad&\quad \begin{matrix}
    4&4&4\\0&4&4\\0&0&2
  \end{matrix}\quad\\ \hline
  \quad 5 \quad&\quad \begin{split}
    &2^21,23,32,41,5
  \end{split} \quad&\quad \begin{matrix}
    1&2&6&6&4\\0&1&2&2&2\\0&0&1&2&2\\0&0&0&1&2\\0&0&0&0&1
  \end{matrix}\quad&\quad \begin{matrix}
    8&8&8&8&8\\0&4&4&0&4\\0&0&4&4&4\\0&0&0&4&4\\0&0&0&0&2
  \end{matrix}\quad\\ \hline
  \quad 6 \quad&\quad \begin{split}
    &2^3,231,24,321,3^2,\\
    &42,51,6
  \end{split} \quad&\quad \begin{matrix}
    1&2&2&4&8&12&8&4\\0&1&2&2&6&8&6&4\\0&0&1&0&2&2&2&2\\
    0&0&0&1&2&6&6&4\\0&0&0&0&1&2&2&2\\0&0&0&0&0&1&2&2\\
    0&0&0&0&0&0&1&2\\0&0&0&0&0&0&0&1
  \end{matrix} \quad&\quad \begin{matrix}
    8&8&8&8&8&8&8&8\\0&8&8&8&8&0&8&8\\0&0&4&0&4&0&0&4\\
    0&0&0&8&8&8&8&8\\0&0&0&0&4&4&0&4\\0&0&0&0&0&4&4&4\\
    0&0&0&0&0&0&4&4\\0&0&0&0&0&0&0&2
  \end{matrix}\quad\\ \hline
\end{array}\]

Next we give the following example to show that the recursive formula (\ref{rec}) is quite efficient for calculating $M_n(\sq,Q)$.
\begin{exam}
For $n=2$:
\[\sq_2=Q_2=2\Pi_2.\]

\noindent For $n=3$:
\[\begin{array}{l}
\sq_3=Q_3=2\Pi_3,\\
\sq_{21}=\sq_2Q_1-2\sq_3=Q_{21}-2Q_3=4\Pi_{21}.
\end{array}\]

\noindent For $n=4$:
\[\begin{array}{l}
\sq_4=Q_4=2\Pi_4,\,\sq_{31}=\sq_3Q_1-2\sq_4=Q_{31}-2Q_4=4\Pi_{31},\\
\sq_{2^2}=\sq_2Q_2-2\sq_3Q_1+2\sq_4=Q_{2^2}-2Q_{31}+2Q_4=4(\Pi_{2^2}-\Pi_{31}).
\end{array}\]

\noindent For $n=5$:
\[\begin{array}{l}
\sq_5=Q_5=4\Pi_5,\,\sq_{41}=\sq_4Q_1-2\sq_5=Q_{41}-2Q_5=4\Pi_{41},\\
\sq_{32}=\sq_3Q_2-2\sq_4Q_1+2\sq_5=Q_{32}-2Q_{41}+2Q_5=4(\Pi_{32}-\Pi_{41}),\\
\sq_{23}=\sq_2Q_3-2\sq_3Q_2+2\sq_4Q_1-2\sq_5=Q_{23}-2Q_{32}+2Q_{41}-2Q_5
=4(\Pi_{23}-\Pi_{32}),\\
\sq_{2^21}=\sq_{2^2}Q_1-2(\sq_{23}+\sq_{32})
=Q_{2^21}-2Q_{23}-2Q_{32}+2Q_{41}=8(\Pi_{2^21}-\Pi_{32}+\Pi_{41}).
\end{array}\]

\noindent For $n=6$:
\[\begin{array}{l}
\sq_6=Q_6=4\Pi_6,\,\sq_{51}=\sq_5Q_1-2\sq_6=Q_{51}-2Q_6=4\Pi_{51},\\
\sq_{42}=\sq_4Q_2-2\sq_5Q_1+2\sq_6=Q_{42}-2Q_{51}+2Q_6=4(\Pi_{42}-\Pi_{51}),\\
\begin{split}
\sq_{3^2}&=\sq_3Q_3-2\sq_4Q_2+2\sq_5Q_1-2\sq_6\\
&=Q_{3^2}-2Q_{42}+2Q_{51}-2Q_6
=4(\Pi_{3^2}-\Pi_{42}),
\end{split}\\
\begin{split}
\sq_{321}&=\sq_{32}Q_1-2(\sq_{42}+\sq_{3^2})
=Q_{321}-2Q_{42}-2Q_{3^2}+2Q_{51}\\
&=8(\Pi_{321}-\Pi_{42}+\Pi_{51}),
\end{split}\\
\begin{split}
\sq_{24}&=\sq_2Q_4-2\sq_3Q_3+2\sq_4Q_2-2\sq_5Q_1+2Q_6\\
&=Q_{24}-2Q_{3^2}+2Q_{42}-2Q_{51}+2Q_6
=4(\Pi_{24}-\Pi_{3^2}),
\end{split}\\
\begin{split}
\sq_{231}&=\sq_{32}Q_1-2(\sq_{42}+\sq_{3^2})
=Q_{231}-2Q_{321}+2Q_{3^2}-2Q_{24}+4Q_{42}-2Q_{51}\\
&=8(\Pi_{231}-\Pi_{321}-\Pi_{3^2}+\Pi_{42}).
\end{split}\\
\begin{split}
\sq_{2^3}&=\sq_{2^2}Q_2-2(\sq_{32}+\sq_{23})Q_1+2(\sq_{42}+\sq_{24})
+4\sq_{3^2}
=Q_{2^3}+2(Q_{24}-Q_{231})\\
&=8(\Pi_{2^3}-\Pi_{231}-\Pi_{321}+\Pi_{42}-\Pi_{51}).
\end{split}
\end{array}\]
\end{exam}

\subsection{Quasisymmetric Schur P-functions}

With respect to $[\cdot,\cdot]$, the dual of Schur's Q-functions are Schur's P-functions, denoted by $\mS'_\la$ (the usual notation $P_\la$ is given up to avoid confusion). Then $\mS_\la=2^{\ell(\la)}\mS'_\la$ for any strict partition $\la$. Let $\sq_\al^*$ be the dual of $\sq_\al$ in $\mB$ with respect to the pairing
$[\cdot,\cdot]$. We call them the \textit{quasisymmetric Schur P-functions}, since they are nice refinements of Schur's P-functions stated as follows.

\begin{theorem}
For any strict partition $\la$, we have
\beq\label{ref}
\mS'_\la=\sum_{\al\in\pp\mc\atop \tilde{\al}=\la}(-1)^{\ell(\si(\al))}\sq^*_\al,
\eeq
where $\tilde\al$ is the unique partition as the rearrangement of $\al$, and $\si(\al)$ is the permutation for this rearrangement.
\end{theorem}
\begin{proof} Let $\mS'_\la=\sum_{\be\in\pp\mc}c_{\la,\be}\sq^*_\be$. Then by (\ref{bi}), for any peak composition $\be$,
\[c_{\la,\be}=[\sq_\be,\mS'_\la]=[\pi(\sq_\be),\mS'_\la]
=[\mS_\be,\mS'_\la]=\de_{\tilde{\be},\la}
(-1)^{\ell(\si(\be))}.\]
\end{proof}
\begin{rem}
In \cite{Ste}, the author also gave the following refinement of Schur's Q-functions in terms of peak functions:
\beq\label{ref}\mS_\la=\sum_{T\in\bs{ShT}(\la)}K_{\La(T)},\eeq
where $\mb{ShT}(\la)$ is the set of standard shifted tableaux of shape $\la$, and $\La(T)$ is the peak set of  $T$.

Let $\bar{\sq}_\al=2^{-\ell(\al)}\sq_\al,\,
\bar{\sq}^*_\al=2^{\ell(\al)}\sq^*_\al,\,\al\in\pp\mc$. Then $\pi(\bar{\sq}_\al)=2^{-\ell(\al)}\mS_\al=\mS'_\al$ and
\beq\label{ref'}\mS_\la=\sum_{\al\in\pp\mc\atop \tilde{\al}=\la}(-1)^{\ell(\si(\al))}\bar{\sq}^*_\al.\eeq
Therefore we call $\bar{\sq}^*_\al$'s the \textit{quasisymmetric Schur Q-functions}.
\end{rem}

\begin{exam}
Let $K_\al=K_{D(\al)},\,\al\in\pp\mc$. Denote by $M_n(\Pi,\bar{\sq})$ and $M_n(\bar{\sq}^*,K)$ the corresponding transition matrices. Then $M_n(\Pi,\bar{\sq})=M_n(\bar{\sq},\Pi)^{-1}=M_n(\bar{\sq}^*,K)^T$, and we have

\[\begin{array}{|c|c|c|c|}
 \hline
  \quad n \quad&\quad \mb{index}\quad&\quad M_n(\bar{\sq},\Pi)\quad & M_n(\Pi,\bar{\sq})\\ \hline
  \quad 3 \quad&\quad 21,3 \quad&\quad \begin{matrix}
    1&0\\0&1
  \end{matrix}\quad&\quad \begin{matrix}
    1&0\\0&1
  \end{matrix}\quad\\ \hline
  \quad 4 \quad&\quad 2^2,31,4 \quad&\quad \begin{matrix}
    1&-1&0\\0&1&0\\0&0&1
  \end{matrix} \quad&\quad \begin{matrix}
    1&1&0\\0&1&0\\0&0&1
  \end{matrix}\quad\\ \hline
  \quad 5 \quad&\quad \begin{split}
    &2^21,23,32,41,5
  \end{split} \quad&\quad \begin{matrix}
    1&0&-1&1&0\\0&1&-1&0&0\\0&0&1&-1&0\\0&0&0&1&0\\0&0&0&0&1
  \end{matrix}\quad&\quad \begin{matrix}
    1&0&1&0&0\\0&1&1&1&0\\0&0&1&1&0\\0&0&0&1&0\\0&0&0&0&1
  \end{matrix}\quad\\ \hline
  \quad 6 \quad&\quad \begin{split}
    &2^3,231,24,321,3^2,\\
    &42,51,6
  \end{split} \quad&\quad \begin{matrix}
    1&-1&0&-1&0&1&-1&0\\0&1&0&-1&-1&1&0&0\\0&0&1&0&-1&0&0&0\\
    0&0&0&1&0&-1&1&0\\0&0&0&0&1&-1&0&0\\0&0&0&0&0&1&-1&0\\
    0&0&0&0&0&0&1&0\\0&0&0&0&0&0&0&1
  \end{matrix} \quad&\quad \begin{matrix}
    1&1&0&2&1&1&0&0\\0&1&0&1&1&1&0&0\\0&0&1&0&1&1&1&0\\
    0&0&0&1&0&1&0&0\\0&0&0&0&1&1&1&0\\0&0&0&0&0&1&1&0\\
    0&0&0&0&0&0&1&0\\0&0&0&0&0&0&0&1
  \end{matrix}\quad\\ \hline
\end{array}\]
\end{exam}

Set $\bar{Q}_\al=2^{-\ell(\al)}Q_\al$, then any $\bar{Q}_\al$ has a positive, integral and unitriangular expansion in $\bar{\sq}_\be$'s by formula (\ref{exp'}) and Lemma \ref{da}. Combining with (\ref{qp}), we know that $M(\bar{\sq},\Pi)$ is also integral and unitriangular.
From the above example, we expect the positivity of $M(\Pi,\bar{\sq})$ (or dually $M(\bar{\sq}^*,K)$), which is more interesting for us.
\begin{con}\label{conj}
$\{\Pi_\al\}_{\al\in\pp\mc}$ has a positive, integral and unitriangular expansion in  $\{\bar{\sq}_\al\}_{\al\in\pp\mc}$. Dually, $\{\bar{\sq}^*_\al\}_{\al\in\pp\mc}$ has such an expansion in the peak functions $\{K_\al\}_{\al\in\pp\mc}$.
\end{con}

Finally we discuss the expansion of quasisymmetric Schur P-functions in terms of the monomial $M_\al$'s or the fundamental $F_\al$'s.
 First note that
%by the dual of the expansion formula (\ref{exp'}), we have
%\beq\label{dex}\sq^*_\al=\sum_{\be\in\pp\mc
%\atop\be\leq_l\al}\lb\sum_{T\in\bs{PCT}(\al,\be)}
%2^{p(T)-m(T)}\rb Q^*_\be\eeq
%for any $\al\in\pp\mc$, where $Q^*_\be$ is the dual of $Q_\be$ in $\mB$. Meanwhile,
\[\begin{split}
\sq^*_\al&=\sum_{\be\in\mc}\lan H_\be,\sq^*_\al\ran M_\be=
\sum_{\be\in\mc}[Q_\be,\sq^*_\al] M_\be\\
&=\sum_{\be\in\mc}\lan R_\be,\sq^*_\al\ran F_\be=
\sum_{\be\in\mc}[\Te(R_\be),\sq^*_\al] F_\be.
\end{split}\]
By formula (\ref{exp'}), the coefficient
\[[Q_\be,\sq^*_\al]=
[\sum_{\ga\in\pp\mc
\atop\be\leq_l\ga}\lb\sum_{T\in\bs{PCT}(\ga,\be)}
2^{p(T)-m(T)}\rb \sq_\ga,\sq^*_\al]=\sum_{T\in\bs{PCT}(\al,\be)}
2^{p(T)-m(T)},\]
which vanishes unless $\be\leq_l\al$. Furthermore, since $\Te(R_\be)=\sum_{\ga\geq\be}(-1)^{\ell(\be)-\ell(\ga)}Q_\ga$ and $\ga\geq\be\Rightarrow\ga\geq_l\be$,  $[\Te(R_\be),\sq^*_\al]=\sum_{\ga\geq\be}
(-1)^{\ell(\be)-\ell(\ga)}[Q_\ga,\sq^*_\al]=0$ unless $\be\leq_l\al$.

If Conjecture \ref{conj} holds, then by the formula
\[\Te(R_\al)=\sum_{\be\in\pp\mc_n\atop D(\be)\subseteq D(\al)\triangle(D(\al)+1)}2^{\ell(\be)}\Pi_\be,\,\al\in\mc_n\]
in \cite[(6)]{BHT},
the coefficients $[\Te(R_\be),\sq^*_\al],\,\be\in\mc$ are also nonnegative integers.
In summary, we have
\begin{prop}\label{pos}
For any $\al\in\pp\mc$,
\[\sq^*_\al=\sum_{\be\leq_l\al}\lb\sum_{T\in\bs{PCT}(\al,\be)}
2^{p(T)-m(T)}\rb M_\be.\]
If Conjecture \ref{conj} holds, then  $\sq^*_\al$ also
has a positive, integral expansion in $F_\be$'s with $\be\leq_l\al$ and the leading coefficient equals to 1.
\end{prop}

For example,
%$Q_{123}=2Q_{24}-Q_{231}+Q_{2^3}$ and
%$\sq^*_{321}=4Q^*_{2^3}+2Q^*_{231}+Q^*_{321}$, thus
%$[Q_{123},\sq^*_{321}]=2$. Similarly, one can compute all the coefficients and get that
\[\begin{split}
\sq^*_{321}&=M_{321}+M_{312}+2(M_{31^3}+M_{231})+4M_{2^3}
+8M_{2^21^2}+2M_{213}+8(M_{2121}+M_{21^22})\\
&+16M_{21^4}+M_{141}+3M_{132}+6M_{131^2}+2M_{123}+8(M_{12^21}
+M_{1212})+16M_{121^3}\\
&+4M_{1^231}+8M_{1^22^2}+16M_{1^221^2}+4M_{1^33}
+16(M_{1^321}+M_{1^42})+32M_{1^6}\\
&=F_{321}+F_{312}+2F_{231}+4F_{2^3}+2(F_{2^21^2}+F_{213})
+3F_{2121}+F_{21^22}+F_{141}\\
&+3F_{132}+2(F_{131^2}+F_{123})+4F_{12^21}+2F_{1212}+F_{1^231}
+F_{1^22^2}
\end{split}\]

%\subsection{Pieri rule for skew operators}
%\begin{prop}
%For $s\geq0$ and for $\al\in\bZ^r$,
%\[K_{\emp_s}^\perp\sq_\al=\sum_{\be\in\bZ^r\atop{\be_i\leq\al_i\atop |\be|=|\al|-s}}2^{\ell(\al-\be)}\sq_\be.\]
%\end{prop}
\bigskip

\centerline{\bf Acknowledgments}
We would like to thank Weiqiang Wang for stimulating discussions.
NJ thanks the partial support of Simons Foundation grant 198129
and NSFC grant 11271138 during this work.

\bigskip
\bibliographystyle{amsalpha}

\end{document}